\documentclass[10pt,reqno]{amsart}

\usepackage{amssymb, verbatim, mathrsfs}
\usepackage{amsmath}
\usepackage{amscd}
\usepackage{diagrams}
\usepackage{enumitem}
\usepackage{mathtools}
\usepackage[draft=false]{hyperref}

\theoremstyle{plain}
\newtheorem{thm}{Theorem}[section]
\newtheorem*{thm*}{Theorem}
\newtheorem{prop}[thm]{Proposition}
\newtheorem{cor}[thm]{Corollary}
\newtheorem{lem}[thm]{Lemma}

\theoremstyle{definition}
\newtheorem{defn}[thm]{Definition}
\newtheorem{defns}[thm]{Definitions}
\newtheorem{eg}[thm]{Example}
\newtheorem{rem}[thm]{Remark}
\newtheorem{rems}[thm]{Remarks}
\newtheorem*{rem*}{Remark}
\newtheorem*{rems*}{Remarks}
\newtheorem*{note*}{Note}

\renewcommand{\/}{\kern 0.05em}
\newcommand{\<}{\langle}
\renewcommand{\>}{\rangle}
\newcommand{\geqs}{\geqslant}
\newcommand{\leqs}{\leqslant}
\newcommand{\inv}{^{-1}}
\newcommand{\eset}{\emptyset}
\newcommand{\del}{\partial}
\newcommand{\ts}{\textstyle}
\renewcommand{\a}{\alpha}
\renewcommand{\b}{\beta}

\renewcommand{\d}{\delta}
\newcommand{\e}{\epsilon}
\renewcommand{\k}{\kappa}
\newcommand{\lam}{\lambda}
\newcommand{\s}{\sigma}
\newcommand{\vrs}{\varepsilon}
\renewcommand{\v}{\varphi}

\newcommand{\z}{\zeta}
\newcommand{\upchi}{\raise.4ex\hbox{$\chi$}}
\newcommand{\fa}{\mathfrak a}
\newcommand{\fe}{\mathfrak e}
\newcommand{\fg}{\mathfrak g}
\newcommand{\fh}{\mathfrak h}
\newcommand{\fl}{\mathfrak l}
\newcommand{\fn}{\mathfrak n}
\newcommand{\fp}{\mathfrak p}
\newcommand{\fs}{\mathfrak s}
\newcommand{\fu}{\mathfrak u}
\newcommand{\fX}{\mathfrak X}
\newcommand{\nil}{\mathfrak{n}\mathfrak{i}\fl}

\newcommand{\A}{\mathcal A}
\newcommand{\C}{\mathcal C}
\newcommand{\E}{\mathcal E}
\renewcommand{\H}{\mathcal H}
\renewcommand{\P}{\mathcal P}
\newcommand{\T}{\mathcal T}
\newcommand{\Z}{\mathcal Z}
\newcommand{\R}{\mathbb R}
\newcommand{\CC}{\mathscr C}
\newcommand{\EE}{\mathscr E}
\newcommand{\HH}{\mathscr H}
\renewcommand{\SS}{\mathscr S}
\newcommand{\VV}{\mathscr V}
\newcommand{\newt}{\scalebox{1.4}{\ensuremath\nu}}

\newcommand{\Evr}{\E^{\/v}_{\/r}}
\newcommand{\tvr}{\tau_r^{v}}
\newcommand{\Rnabla}{R^{\nabla}\!}
\newcommand{\nab}[1]{\nabla\kern-.2em\lower.8ex\hbox{\SMALL$#1$}\/}
\newcommand{\tildenab}[1]{\tilde\nabla\kern-.2em\lower.8ex\hbox{\SMALL$#1$}\/}
\newcommand{\nabv}[1]{\nabla^v\kern-.3em\lower.8ex\hbox{\SMALL$#1$}\/}
\newcommand{\nabsq}[2]{\nabla^2\kern-.3em\lower.8ex\hbox{\SMALL$#1,#2$}\/}
\newcommand{\shape}[1]{\A\kern-.05em\lower.8ex\hbox{\SMALL$#1$}\/}
\newcommand{\tr}[1]{\trace\kern-.05em\lower.8ex\hbox{\SMALL$#1$}\kern.05em}
\newcommand{\diffop}[2]{#1\kern-.1em\lower.8ex\hbox{\SMALL$#2$}\/} 
\newcommand{\uperp}[1]{#1\raise1.1ex\hbox{\SMALL$\bot$}}
\newcommand{\abs}[1]{\lvert{#1}\rvert}
\newcommand{\norm}[1]{\lVert{#1}\rVert}

\DeclareMathOperator{\trace}{trace} 
\DeclareMathOperator{\rank}{rank}
\DeclareMathOperator{\volume}{vol} 
\DeclareMathOperator{\Ric}{Ric} 
\DeclareMathOperator{\diverge}{div}

\numberwithin{equation}{section}
\allowdisplaybreaks

\begin{document}
\title[Higher-power harmonic maps and sections]
{Higher-power harmonic maps and sections}

\author{A. Ramachandran}

\author{C.~M. Wood}
\address{Department of Mathematics \\
University of York \\
Heslington, York Y010 5DD\\
U.K.} \email{chris.wood@york.ac.uk}

\keywords{Higher-power energy, higher-power harmonic map, $r$-conformal map, higher-power harmonic section, $r$-horizontal section, Newton tensor, twisted skyrmion, Riemannian vector bundle, $r$-parallel section, sphere subbundle, $3$-dimensional unimodular Lie group, left-invariant metric, invariant (unit) vector field, Milnor map, principal Ricci/sectional curvatures.}

\subjclass[2010]{53C05, 53C43, 58E20, 58E30}

\date{\today}

\dedicatory{In memoriam Anand.}

\maketitle

\begin{abstract}
The variational theory of higher-power energy is developed for mappings between Riemannian manifolds, and more generally sections of submersions of Riemannian manifolds, and applied to sections of Riemannian vector bundles and their sphere subbundles.  A complete classification is then given for left-invariant vector fields on 3-dimensional unimodular Lie groups equipped with an arbitrary left-invariant Riemannian metric.
\end{abstract}

\section{Introduction}
In their definitive paper on harmonic mappings of Riemannian manifolds \cite{ES}, Eells and Sampson remark: ``Although the present work is devoted primarily to the (energy) functional $E$ and its extremals, there will be the indications that we will want ultimately to consider other types of energy of maps.''  Interestingly, and to the best of our knowledge unbeknown to the authors of \cite{ES}, a couple of years earlier Skyrme had used one of those ``other types of energy'' as a constituent of a mathematically tractable model for the nucleon \cite{Sk}.  The `skyrmion' model has since gained popularity in the physics and mathematical physics communities; see for example \cite{GHM, HMS, Sl, SS, Z}.  However, despite this, and intense interest in the theory of harmonic maps, so far there has been little work on establishing a general theory of the ``other types of energy'' mentioned in \cite{ES}.  In this paper we will make some moves in this direction.   

\par
Let $\v\colon(M,g)\to(N,h)$ be a smooth mapping of Riemannian manifolds, with $\dim M=m$.  Viewing the $2$-tensors $g$ and $\v^*h$ as morphisms $TM\to T^*M$, the energy of $\v$ is obtained by integrating (globally, or locally, depending on whether $M$ is compact) the trace of the tensor field $\a=g\inv\v^*h$ (sometimes known as the \textsl{Cauchy-Green tensor} \cite{Sl}).  The \textsl{higher-power energies} of $\v$ are obtained by integrating the higher-degree elementary invariants of $\a$, the terminology reflecting that these correspond to higher exterior powers of $d\v$.  Geometrically, therefore, they measure the extent to which $\v$ deforms higher-dimensional infinitesimal volume.  In particular, the highest-power energy measures the deformation by $\v$ of full volume, and although it does not coincide with the volume functional has the same critical points (Remark \ref{rem:volume}); cf.\ the relationship between energy and length when $m=1$.  Furthermore, higher-power energies allow the generalisation to higher dimensions of some well-known properties of standard energy when $m=2$, such as conformal invariance and area majorisation (Propositions \ref{prop:conf} and \ref{prop:major}, Corollary \ref{cor:confmajor}, Example \ref{eg:calibrated}).

\par
More generally, let $\pi\colon (P,k)\to (M,g)$ be a smooth submersion of Riemannian manifolds, and let $\s\colon M\to P$ be a smooth section.  At the outset, the Riemannian metrics $g,k$ are not assumed to have any compatibility properties; for example, $\pi$ is not necessarily a Riemannian submersion \cite{ON}, or semi-conformal as in the theory of harmonic morphisms \cite{F}.  Nor is $\pi$ assumed to be a fibre bundle, or locally trivial, and there are no assumptions on the geometry of its fibres.  Nevertheless it is possible to define \textsl{higher-power vertical energies,} as measurements of the total twisting of $\s$, in the sense of deviation from horizontality.  The details of these constructions are presented in \S\ref{sec:hp}, where we also introduce the families of \textsl{Newton tensors} associated to $\v$ and $\s$; these play an important r\^ole in the subsequent variational theory.  

\par
In \S\ref{sec:firstvar} we derive the Euler-Lagrange equations for the higher-power energy functionals, and thereby characterise \textsl{higher-power harmonic maps.}  For any integer $r=1,\dots,m$ we refer to the critical points of the $r$-th higher-power energy (ie.\ the functional obtained from the $r$-th elementary invariant of $\a$) as \textsl{$r$-power harmonic maps,} or more briefly \textsl{$r$-harmonic maps.} (We are aware of the overlap in terminology with the, somewhat different, variational theory of `$p$-energy' \cite[Chapter 3]{EL2}, \cite{BG}, \cite{Mis}, with which we shall not be concerned.)  We refer to the Euler-Lagrange operator $\tau_r(\v)$ as the \textsl{$r$-th tension field} of $\v$ (Definition \ref{defn:rtension}).  In contrast to \cite{ES}, in this paper we make no attempt to explore the analytic properties of $\tau_r(\v)$ (such as existence of solutions, unique continuation, etc.), and simply remark that for $r>1$ it is a second order quasi-linear semi-elliptic partial differential operator.

\par
The $r$-harmonic map equations (Theorem \ref{thm:rhamap}) emerge as a corollary (Example \ref{eg:srhamap}) of more general computations for the first variation of higher-power vertical energy with respect to variations through sections (Theorems \ref{thm:firstvar1} and \ref{thm:firstvar2}).  The Euler-Lagrange equations for this constrained variational problem characterise what we choose to call \textsl{$r$-harmonic sections.}  Our derivation is analogous to the coordinate-free approach for harmonic maps used in \cite{EL1}, albeit  more subtle.  Specifically, the analogue in this context of the second fundamental form (of a mapping) is no longer symmetric, its asymmetry being a manifestation of the \textsl{curvature} of the submersion $\pi$ (Definition \ref{defn:curvform}).  This notion of curvature, although somewhat unorthodox at this level of generality, encodes essentially the same geometric and topological information as the second fundamental form of a Riemannian submersion (Proposition \ref{prop:curvformgeom}, cf.\ Lemma \ref{lem:fundform}); it coincides with the standard definition of curvature when, for example, $\pi$ is a vector bundle with linear connection (Remarks \ref{rems:can}).  The appearance of curvature in the $r$-harmonic section equations (Theorem \ref{thm:firstvar2}) expresses the interaction of the section with the extrinsic geometry of the fibres of $\pi$.  Although in general unrealistic to assume that $\pi$ is flat (ie.\ its curvature vanishes; cf.\ Proposition \ref{prop:curvformgeom}, Theorem \ref{thm:3dflat}), if $\pi$ has totally geodesic (t.g.) fibres the equations simplify to the vanishing of the \textsl{$r$-th vertical tension field} of $\s$ (Definition \ref{defn:rvtension}); this is the case for most commonly encountered examples, such as fibre bundles with Kaluza-Klein geometry (Remarks \ref{rems:firstvar2}).

\par
It is also of interest to study sections that are $r$-harmonic maps.  Since these are critical points of an unconstrained variational problem, the Euler-Lagrange operator $\tau_r(\s)$ acquires a horizontal component, which may be deconstructed when $\pi$ is a Riemannian submersion with t.g.\  fibres (Theorem \ref{thm:flatharm}).  Interestingly, the curvature of $\pi$ also appears here, simplifying this time when $\s$ is a \textsl{flat section} (Definition \ref{defn:curvform}), a condition that generalises flatness of $\pi$, and horizontality of $\s$ (Remarks \ref{rems:curvform}), although still somewhat restrictive (Theorem \ref{thm:3dflat}).  Under the same hypotheses (viz.\ $\pi$ a Riemannian submersion with t.g.\ fibres) the vertical component of $\tau_r(\s)$ fragments into a linear combination of the $r$-th vertical tension field  and those of lower degree (Theorem \ref{thm:hasec}), prompting the definition of a \textsl{twisted $r$-skyrmion} (Definition \ref{defn:twistskyr}).  It is notable that both components of $\tau_r(\s)$ involve the divergence of a relevant Newton tensor, which, in contrast to the Newton tensors associated to the shape operator of an isometric immersion or the covariant Hessian of a smooth function \cite{R1, R2, R3}, does not in general vanish (see for example Theorem \ref{thm:newton} and Lemma \ref{lem:vnewt2}); indeed,  for mappings, the divergence of the $(r-1)$-st Newton tensor is the principle obstruction to a totally geodesic map being $r$-harmonic (Remark \ref{rem:reilly}).  We conclude Section \ref{sec:firstvar} with a characterisation of $r$-harmonicity for flat sections, and in particular a synopsis of the horizontal case (Theorem \ref{thm:harmhoriz}).

\par
In \S\ref{sec:vbundles} the results of \S\ref{sec:firstvar} are interpreted for Riemannian vector bundles, equipped with their natural (Sasaki) geometry.  We prove two rigidity results: for bundles with compact base (Theorem \ref{thm:compactvb}), and for sections of constant length (Corollary \ref{cor:hphasecsb}).  These generalise phenomena that are familiar when $r=1$ \cite{I, N, W1}: `rigid' sections are those that are \textsl{$r$-parallel} (Definition \ref{defn:rpar}), becoming more `flexible' as $r$ increases (see for example Theorem \ref{thm:3dzero}).   For a section of constant length, rigidity can be mitigated by working within the corresponding sphere subbundle, a proven methodology when $r=1$ \cite{Gil}; intuitively, reducing the class of variations increases the likelihood of critical points.  In this case, higher-power energies may be regarded as measurements of `total bending' \cite{Wie}. The Euler-Lagrange equations for this restricted variational problem are obtained in Theorem \ref{thm:hphasecsb}, the `untwisted version' of which yields a characterisation of $r$-harmonic maps into spheres (Corollary \ref{cor:hphamapsphere}) generalising \cite{Sm},  a potentially interesting topic for further study.

\par
The primary example of a Riemannian vector bundle is of course the tangent bundle of a Riemannian manifold, equipped with the Levi-Civita connection, and in \S\ref{sec:3dlie} we examine in detail the various facets of the foregoing theory in the context of invariant vector fields $\s$ on a $3$-dimensional Lie group $G$, which for simplicity we assume to be unimodular and endowed with a left-invariant metric.  The algebraic and geometric properties of $G$ were worked out, beautifully, by Milnor in \cite{Mil} (presented here as 'Milnor's list'), using techniques that effectively reduce our computations to classical vector algebra, which as far as possible we try to keep coordinate-free and geometrically cogent.  (Further simplification via the Principle of Symmetric Criticality \cite{Pal} is however not possible, since in the majority of cases $G$ is non-compact.)  Assuming without loss of generality that $\s$ has unit length (Corollary \ref{cor:hphasecusb}), we refer to $r$-harmonic sections of the unit tangent bundle $UG\to G$ as \textsl{$r$-harmonic unit vector fields,} generalising terminology used when $r=1$ \cite{Gil}.  Denoting the set of $r$-harmonic invariant unit fields by $\H_r$, we obtain a highly geometric classification of $\H_2$ (Theorem \ref{thm:2pharmsb}), 
and recover the classification of $\H_1$ obtained in \cite{GDV}, which we recast in a more geometric light (Theorem \ref{thm:1pharmsb}).  \textit{En route} we calculate the first vertical Newton tensor, and observe precisely when it is solenoidal (Theorem \ref{thm:newton}).  We identify the subsets $\Z_r\subseteq\H_r$ of absolute minimum $r$-th vertical energy (viz.\ the $r$-parallel invariant vector fields), which by rigidity are precisely the invariant $r$-harmonic sections of the full tangent bundle $TG\to G$ (Theorem \ref{thm:3dzero}), and observe that $\H_r=\H_{r-1}\cup\Z_r$ for $r=2,3$ (Corollary \ref{cor:harmz}), a feature that we suspect is rather specific to this example.  A comprehensive classification based on Ricci curvature is then presented (Corollary \ref{cor:harmz}), from which we explicitly determine $\H_r$ and $\Z_r$ for each geometry on Milnor's list.  We then discover which $\s$ are twisted skyrmions (Theorem \ref{thm:uniskyrme}), and use this to classify the invariant unit fields that are $r$-harmonic maps $G\to UG$ for $r=1,2,3$ (Theorem \ref{thm:hpharmaps}).  When $r=3$ these are precisely the invariant minimal immersions, and we recover results of \cite{TV}. 

\par
The paper is a development of research carried out in \cite{Ram, W1}.  All manifolds, mappings, bundles etc.\ are assumed smooth ($\CC^\infty$), and to simplify notation, with one notable exception all connections are denoted by (undecorated) $\nabla$.

\section{Higher-power energy and Newton tensors}
\label{sec:hp}
Let $A$ be a $m\times m$ (real) matrix, with characteristic polynomial:
\begin{align*}
\upchi_A(\lam)
&=\det(A-\lam 1)
=\sum_{k=0}^m (-1)^k\vrs_{m-k}(A)\/\lam^k.
\end{align*}
Thus $\vrs_0(A)=1$, $\vrs_1(A)=\trace(A)$, $\vrs_m(A)=\det(A)$, and in general the \textsl{$r$-th elementary invariant} $\vrs_r(A)$ is a $GL_m$-invariant homogeneous  polynomial of degree $r$, the sum of the leading $r\times r$ minors of $A$.  The elementary invariants may be characterised recursively; for example:
\begin{align}
\vrs_2(A)
&=\textstyle\sum_{i<j}(A_{ii}A_{jj}-A_{ij}A_{ji})
=\tfrac12\textstyle\sum_{i,j}(A_{ii}A_{jj}-A_{ij}A_{ji}) 
\notag \\[0.5ex]
&=\textstyle\tfrac12(\trace(A)^2-\trace(A^2)),
\label{eq:sigma2}
\end{align}
and in general by the \textsl{Newton-Girard identity} \cite[p.~81]{VW}:
\begin{equation}
r\/\vrs_r(A)
=\sum_{k=1}^r (-1)^{k-1} \vrs_{r-k}(A)\,\vrs_1(A^k).
\label{eq:sigmar}
\end{equation}
The $r$-th \textsl{Newton polynomial} of $A$ is the following interpolant of $\upchi_A(\lam)$:
\begin{equation}
\upchi_r(\lam)
=\upchi_{A,r}(\lam)
=\sum_{k=0}^r (-1)^k\vrs_{r-k}(A)\/\lam^k. 
\label{eq:newtpoly}
\end{equation}
Thus $\upchi_0(\lam)=1$, $\upchi_1(\lam)=\trace(A)-\lam$, and in general $\upchi_r(\lam)$ may be characterised recursively by:
\begin{equation}
\upchi_{r}(\lam)=\vrs_r(A)-\lam\/\upchi_{r-1}(\lam).
\label{eq:newtrec}
\end{equation}  

\par
By $GL_m$-invariance, $\vrs_r(\a)$ and $\upchi_{\a,r}(\lam)$ are defined for any linear endomorphism $\a$ of an $m$-dimensional (real) vector space $V$.  Evaluation of the Newton polynomials (of $\a$) at $\a$ yields an associated family of \textsl{Newton endomorphisms} $\upchi_r(\a)\colon V\to V$.  In particular, $\upchi_0(\a)=1_V$, $\upchi_1(\a)=\trace(\a)1_V-\a$, and $\upchi_m(\a)=0$ by the Cayley-Hamilton Theorem.  For future use we record the following identities.  

\begin{lem}\label{lem:char}
For all integers $r=1,\dots,m${\rm:}

\begin{itemize}[leftmargin=1.75em, itemsep=1ex]
\item[\rm i)]
$\vrs_r(1+\a)
=\vrs_r(\a)+(m-r+1)\vrs_{r-1}(\a)+\cdots
+\textstyle\binom{m-1}{r-1}\vrs_1(\a)+\textstyle\binom{m}{r}$.

\item[\rm ii)]
$\upchi_r(1+\a)
=\upchi_r(\a)+(m-r)\upchi_{r-1}(\a)+\cdots
+\textstyle\binom{m-2}{r-1}\upchi_1(\a)
+\textstyle\binom{m-1}{r}1_V$.

\item[\rm iii)]
$\upchi_{c\a,r}(c\/\a)=c^r\/\upchi_{\a,r}(\a)$,
for all $c\in\R$.
\end{itemize}
\end{lem}

\begin{proof}
(i) follows by expanding the characteristic polynomial 
$\upchi_{1+A}(\lam)=\upchi_A(\lam-1)$ and inspecting coefficients, where $A$ is any matrix representing $\a$.  (ii) and (iii)  follow inductively from \eqref{eq:newtrec}, using the homogeneity of the elementary invariants to obtain (iii).
\end{proof}

We record two further useful properties of Newton endomorphisms,  the first of which is an immediate consequence of Newton's identity \eqref{eq:sigmar}.

\begin{lem}\label{lem:trace}
For all $r=1,\dots,m${\rm:}
$$
\trace(\a\circ\upchi_{r-1}(\a))=r\/\vrs_r(\a).
$$
\end{lem}

Secondly, suppose $\a_t$ is a differentiable 1-parameter family of endomorphisms of $V$, with `variation vector' $\b_t=d\a_t/dt$ (another endomorphism).  The elementary invariants of $\a_t$ are then differentiable $\R$-valued functions (of $t$).  Clearly:
$$
\frac{d}{dt}\/\vrs_1(\a_t)
=\trace(\b_t)=\trace(\b_t\circ\upchi_0(\a_t)),
$$
and by \eqref{eq:sigma2}:
\begin{align*}
\frac{d}{dt}\/\vrs_2(\a_t)
&=\trace(\a_t)\trace(\b_t)-\trace(\b_t\circ\a_t) 
=\trace(\b_t\circ\upchi_1(\a_t)).
\end{align*}
In general:

\begin{lem}\cite[Lem.\ A]{R1}\label{lem:eldiff}
For all $r=1,\dots,m${\rm:}
$$
\frac{d}{dt}\,\vrs_r(\a_t)
=\trace (\b_t\circ\upchi_{r-1}(\a_t)).
$$
\end{lem}

Now let $\v\colon(M,g)\to(N,h)$ be a smooth mapping of Riemannian manifolds, and let $\a$ be the self-adjoint $(1,1)$-tensor on $M$ metrically dual to $\v^*h$:
\begin{equation}
g(\a(X),Y)=\v^*h(X,Y),
\label{eq:cg}
\end{equation}
for all $X,Y\in\fX(M)$.
The elementary invariants of $\a$, computed pointwise on tangent spaces, define smooth $\R$-valued functions on $M$, which, being dependent primarily on $\v$ (assuming $g$ and $h$ are fixed) will be denoted $\vrs_r(\v)\colon M\to\R$ for all $r=1,\dots,m=\dim M$.  Likewise, the Newton endomorphisms of $\a$ define a family of self-adjoint $(1,1)$-tensors on $M$, which we denote by $\nu_r(\v)$ and refer to  as the \textsl{Newton tensors} of $\v$.  If $\{e_i\}$ is a local orthonormal tangent frame of $(M,g)$ then:
\begin{align*}
\vrs_1(\v)=\ts\sum_i g(\a(e_i),e_i)
=\ts\sum_i h(d\v(e_i),d\v(e_i))
=\lVert d\v\rVert^2.
\end{align*}
Furthermore by \eqref{eq:sigma2} and \eqref{eq:cg}:
\begin{align*}
\vrs_2(\v)
&=\tfrac12\ts\sum_{i,j}
\bigl(\abs{d\v(e_i)}^2 \abs{d\v(e_j)}^2
-h(d\v(e_i),d\v(e_j))^2\bigr) 
\label{eq:sigma2phi} \\[0.5ex]
&=\tfrac12\ts\sum_{i,j}
\abs{d\v(e_i)\wedge d\v(e_j)}^2 
\notag \\[0.5ex]
&=\ts\sum_{i<j}
\abs{(d\v\wedge d\v)(e_i\wedge e_j)}^2 \\
&=\norm{d\v\wedge d\v}^2,
\notag
\end{align*}
and by Newton's identity \eqref{eq:sigmar}:
\begin{equation}
\vrs_r(\v)=\norm{(d\v)^r}^2,
\label{eq:extpower}
\end{equation}
where:
$$
(d\v)^r=d\v\wedge\cdots\wedge d\v,
$$ 
the $r$-th exterior power, viewed as an $r$-form on $M$ with values in $\wedge^r(\v\inv TN)$, where $\v\inv TN\to M$ is the pullback bundle.  Thus $\vrs_r(\v)$ may be interpreted geometrically as the average infinitesimal distortion by $\v$ of squared $r$-volume, with $\vrs_r(\v)_x=0$ precisely when $\rank d\v_x<r$.  Observing that
$\tfrac12\/\vrs_1(\v)$ is the energy density of $\v$, we define the \textsl{$r$-th higher-power energy} of $\v$, or more briefly the \textsl{$r$-th energy,} by:
$$
\E_r(\v)
=\frac12\int_M\vrs_r(\v)\volume(g),
$$
for all integers $r=1,\dots,m$, assuming for convenience that $M$ is compact.  Then $\E_r$ is non-negative, and its zeroes are precisely the mappings $\v$ with $\rank\v<r$ everywhere.  Critical points of $\E_r$ will be called \textsl{$r$-power harmonic maps,} or more briefly \textsl{$r$-harmonic maps.} 

\begin{rem}
\label{rem:volume}
Since $\vrs_m(\v)$ is the square of the volume density $v(\v)$, unless $\v$ is an isometric immersion $\E_m(\v)$ is not the volume of $\v$.  In particular, unlike the volume functional, $\E_m$ depends on both $g$ and $h$.  Nevertheless the two functionals have the same critical points; see for example Theorem \ref{thm:hpharmaps}.
\end{rem}

It has been known \textit{ab initio} \cite[p.\ 126]{ES} that when $m=2$ the energy functional is conformally invariant; that is, dependent only on the conformal structure of the domain.  Higher-power energies exhibit this in higher dimensions. 

\begin{prop}\label{prop:conf}
If $m=2r$ then $\E_r(\v)$ is conformally invariant.  The converse holds provided $\rank \v\geqs r$ somewhere.
\end{prop}

\begin{proof}
For clarity in this context, we write $\vrs_r(\v)=\vrs_r(\v,g)$ and $\E_r(\v)=\E_r(\v,g)$.  If $\rho\colon M\to\R^+$ is a continuous  function with $\rho^2$ smooth, then by the homogeneity of the elementary invariants:
$$
\vrs_r(\v,\rho^2 g)=\rho^{-2r} \vrs_r(\v,g).
$$
Since the volume element transforms:
$$
\volume(\rho^2 g)=\rho^m \volume(g),
$$
it follows that:
\begin{equation}
\E_r(\v,\rho^2 g)-\E_r(\v,g)
=\frac12\int_M \bigl(\rho^{m-2r}-1\bigr) \vrs_r(\v,g) \volume(g),
\label{eq:conf}
\end{equation}
which vanishes if $m=2r$.  Conversely, if $\rank \v\geqs r$ somewhere then $\vrs_r(\v,g)>0$ on an open set, so if \eqref{eq:conf} vanishes for all $\rho$ then certainly $m-2r=0$.
\end{proof}

It is also well-known that when $m=2$ energy majorises area, with equality precisely for mappings that are weakly conformal \cite[p.\ 126]{ES}, \cite{L}.  Again, higher-power energies generalise this to higher dimensions.  

\begin{defn}\label{defn:rconf}
A mapping $\v$ is \textsl{$r$-conformal} if $\v$ is conformal on an open subset, away from which $\rank \v<r$; thus $1$-conformality is equivalent to weak conformality.
\end{defn}

\begin{prop}\label{prop:major}
If $m=2r$ then $\vrs_r(\v)\geqs \binom{m}{r}\/v(\v)$, with equality precisely when $\v$ is $r$-conformal.
\end{prop}

\begin{proof}
Let $\{e_i\}$ be a local $g$-orthonormal tangent frame diagonalising $\a$, and set $\rho_i=\norm{d\v(e_i)}$; the eigenvalues of $\a$ are therefore $\rho_i^{\;2}$.  Then:
\begin{align*}
0&\leqs
\ts\sum_{\mu\in S_m}
(\rho_{\mu(1)}\cdots \rho_{\mu(r)}
-\rho_{\mu(r+1)}\cdots \rho_{\mu(m)})^2 \\[1ex]
&=(r!)^2\bigl(\vrs_r(\v)-\ts\binom{m}{r}\rho_1\cdots\rho_m\bigr)
=(r!)^2\bigl(\vrs_r(\v)-\ts\binom{m}{r}\/v(\v)\bigr).
\end{align*}
If $\v$ is $r$-conformal then for all $x\in M$ either $\v^*h(x)=\rho(x)^2g(x)$, in which case $\rho_1(x)=\cdots=\rho_m(x)=\rho(x)$, or $\rank d\v_x<r$, in which case at least $r+1$ of the $\rho_i(x)$ vanish.  In either case:
\begin{equation}
\rho_{\mu(1)}(x)\cdots \rho_{\mu(r)}(x)
-\rho_{\mu(r+1)}(x)\cdots \rho_{\mu(m)}(x)=0,
\label{eq:equal}
\end{equation}
for all permutations $\mu\in S_m$, hence $\vrs_r(\v)=\binom{m}{r}\/v(\v)$.  Conversely, given the system \eqref{eq:equal}, if one of the $\rho_i(x)$ vanishes then so do at least $r$ others, whereas if no $\rho_i(x)$ vanishes then all are equal; thus $\v$ is $r$-conformal.
\end{proof}

\begin{cor}\label{cor:confmajor}
Suppose $m=2r$ and $\v\colon M\to (N,h)$ is a $r$-conformal map of a compact conformal manifold $M$.  If $\v$ is a local (resp.\ global) minimiser of volume then $\v$ is a local (resp. global) minimum of $\E_r$ with respect to any Riemannian metric $g$ in the conformal class of $M$; in particular, $\v$ is a $r$-harmonic map.  
\end{cor}

\begin{eg}\label{eg:calibrated}
Let $\v\colon M\to (N,h)$ be a calibrated immersion of a compact $2r$-manifold into a calibrated Riemannian manifold.  Then $\v$ is a minimum of $\E_r$ with respect to any metric on $M$ conformal to $\v^*h$.  
\end{eg}

The Newton tensors of $\v$ may be used to define \textsl{higher-power tension fields:}

\begin{defn}\label{defn:rtension}
The \textsl{$r$-th tension field} of $\v$ is the following section of $\v\inv TN$:
$$
\tau_r(\v)
=\trace\nabla(d\v\circ\nu_{r-1}(\v))
=\textstyle\sum_i\nab{e_i} (d\v\circ\nu_{r-1}(\v))(e_i),
$$
for all $r=1,\dots,m$.
\end{defn}

\begin{rem}
Since $\nu_0(\v)$ is the identity we have: 
$$
\tau_1(\v)=\trace\nabla d\v=\tau(\v),
$$ 
the usual tension field of harmonic map theory \cite{ES}.  
\end{rem}

The following result is a consequence of Theorem \ref{thm:firstvar1} below (see Example \ref{eg:srhamap}).

\begin{thm}\label{thm:rhamap}
A smooth mapping $\v$ of Riemannian manifolds is a $r$-harmonic map if and only if $\tau_r(\v)=0$.
\end{thm}

We now generalise the situation as follows.  Let $\pi\colon(P,k)\to(M,g)$ be a smooth submersion of Riemannian manifolds, with $\dim P=p$.  The tangent bundle of $P$ then splits as an orthogonal direct sum:
$$
TP=\VV\oplus\HH, 
$$
where $\VV=\ker d\pi$, and we refer (as usual) to $\VV$ (resp.\ $\HH$) as the \textsl{vertical} (resp.\ \textsl{horizontal}) distribution.  Thus, every tangent vector $A\in TP$ splits:
\begin{equation}
A=A^v+A^z,
\label{eq:verthoriz}
\end{equation}
where $A^v\in\VV$ and $A^z\in\HH$.  

\par
Suppose $\s$ is a smooth section of $\pi$.  The \textsl{vertical derivative} of $\s$ is:
\begin{equation}
d^v\s(X)=(d\s(X))^v,
\label{eq:vertderiv}
\end{equation}
for all $X\in \fX(TM)$; thus $d^v\s$ may be regarded as a section of $T^*M\otimes\s\inv\VV$.  The horizontal derivative $d^z\s$ is defined similarly.  The \textsl{vertical energy} of $\s$ is then defined:
$$
\E^v(\s)=\frac12\int_M \lVert d^v\s\rVert^2\volume(g).
$$ 
The zeroes of $\E^v$ are precisely the horizontal sections, and stationary points of $\E^v$ with respect to variations through sections are often referred to as  \textsl{harmonic sections;} \cite{W2} \textit{et seq.}   Taking $\a$ now to be the
\textsl{vertical Cauchy-Green tensor} of $\s$:
\begin{equation}
g(\a(X),Y)=k(d^v\s(X),d^v\s(Y)),
\label{eq:vcg}
\end{equation}
and denoting the elementary invariants of $\a$ by $\vrs_r^v(\s)\colon M\to\R$, we define the \textsl{$r$-th vertical energy} of $\s$ by:
\begin{equation}
\Evr(\s)=\frac12\int_M\vrs_r^v(\s)\volume(g),
\label{eq:rvertenergy}
\end{equation}
for all $r=1,\dots,m$.  

\begin{defn}\label{defn:rhoriz}
A section is \textsl{$r$-horizontal} if its vertical derivative has rank everywhere less than $r$
\end{defn}

The zeroes of $\Evr$ are precisely the $r$-horizontal sections; some examples are given in Theorem \ref{thm:3dzero}. It follows that $\Evr$ is trivial for all $r>p-m$, the dimension of the fibres of $\pi$; however since $r\leqs m$ this arises only if $p<2m$ (see also Corollary \ref{cor:hphasecsb}).  Critical points of $\Evr$ with respect to variations through sections will be called \textsl{$r$-power harmonic sections,} or more briefly \textsl{$r$-harmonic sections.}  

\par
The Newton endomorphisms of $\a$ again define a family self-adjoint $(1,1)$-tensors on $M$, which we call the \textsl{vertical Newton tensors} of $\s$ and denote by $\nu_r^v(\s)$.  The \textsl{higher-power vertical tension fields} of $\s$ are then defined:

\begin{defn}\label{defn:rvtension}
The \textsl{$r$-th vertical tension field} of $\s$ is the following section of $\s\inv\VV$:
$$
\tvr(\s)
=\trace\nabla^v(d^v\s\circ\nu_{r-1}^v(\s))
=\ts\sum_i\nabv{e_i}(d^v\s\circ\nu_{r-1}^v(\s))(e_i),
$$
where $\nabla^v$ is the pullback of the linear connection in the vector bundle $\VV\to P$ obtained by orthogonally projecting the Levi-Civita connection of $(P,k)$.  
\end{defn}

\begin{rem}
Since $\nu_0^v(\s)$ is the identity we have: 
$$
\tau^v_1(\s)=\trace\nabla^v d^v\s=\tau^v(\s),
$$ 
the \textsl{vertical tension field} \cite{W3}, whose vanishing characterises harmonic sections in most commonly encountered situations (Theorem \ref{thm:firstvar2} describes the general case).
\end{rem}  

\begin{defns}\label{defn:twistrace}
If $Q$ is a 2-tensor on $M$, possibly vector bundle-valued, and $T$ is a self-adjoint $(1,1)$-tensor, then the \textsl{$T$-twisted trace} of $Q$ is:
$$
\tr{T}\,Q=\ts\sum_i Q(e_i,Te_i)=\ts\sum_i Q(Te_i,e_i).
$$
The \textsl{divergence} of $T$ is the following vector field:
$$
\diverge T=\trace\nabla T
=\ts\sum_i \nab{e_i}T(e_i).
$$
\end{defns}

An elementary calculation allows higher-power tension fields to be expressed as a twisted trace, at the expense of the divergence of the Newton tensor.  

\begin{thm}\label{thm:vsrtension}
For all $r=1,\dots,m$:
$$
\tvr(\s)=\tr{\newt} \nabla^v d^v\s+d^v\s(\diverge \newt),
$$
where $\newt=\nu_{r-1}^v(\s)$.  
\end{thm}

Applying Theorem \ref{thm:vsrtension} to the graph of a mapping yields:

\begin{cor}\label{cor:srtension}
If $\v$ is a smooth map then for all $r=1,\dots,m$:
\begin{equation*}
\tau_r(\v)=\tr{\newt}\nabla d\v+d\v(\diverge \newt),
\label{eq:srtensiontwist}
\end{equation*}
where $\newt=\nu_{r-1}(\v)$.  
\end{cor}

\begin{rems}
\label{rem:reilly}
\item{}
\begin{itemize}[leftmargin=1.4em]
\item[1)]
The Newton tensors of maps/sections are typically not solenoidal; see for example Theorem \ref{thm:newton} and Lemma \ref{lem:vnewt2}.

\item[2)]
If $\v$ is a totally geodesic map then $d\v$ has constant rank \cite{V} so $K(\v)=\ker d\v$ is a vector subbundle of $TM$.  Then $\v$ is a $r$-harmonic map precisely when $\diverge\newt$ is a section of $K(\v)$, where $\newt=\nu_{r-1}(\v)$; thus $\diverge\newt$ is the obstruction modulo $K(\v)$ for $\v$ to be $r$-harmonic.  In particular, a harmonic map need not necessarily be $r$-harmonic for $r>1$.
\end{itemize}
\end{rems}

\section{Higher-power harmonic sections, and curvature}
\label{sec:firstvar}
 
Let $\vartheta^v\colon TP\to\VV$ and $\vartheta^z\colon TP\to\HH$ denote the orthogonal projection morphisms; thus for all $A\in TP$:
$$
\vartheta^v(A)=A^v,\qquad
\vartheta^z(A)=A^z.
$$  
An exterior derivative $d^v$ and coderivative $\d^v$ on the space of $\VV$-valued differential forms on $P$ are obtained from the connection $\nabla^v$ in the bundle $\VV\to P$, in the usual way.  

\par
Now let $\s_t$ be a smooth 1-parameter variation of $\s$ through sections, defined for all $t$ in some open interval $I$ about $0\in\R$, with $\s_0=\s$ and associated homotopy: 
$$
\Sigma\colon M\times I\to N;\,\Sigma(x,t)=\s_t(x).
$$  
Let $V_t$ be the variation field: 
$$
V_t(x)=\frac{\partial\Sigma}{\partial t}\Bigr|_{(x,t)}
=d\/\Sigma(\del_t(x,t)), 
$$
where $\del_t$ is the unit vector field on $M\times I$ in the positive $\R$-direction.  Thus $V_t$ is a section of the pullback bundle $\s_t^{\;-1}\VV\to M$; in particular, $V=V_0$ is a section of $\s\inv\VV$.  The \textsl{vertical second fundamental form} of $\Sigma$ is defined
$$
\nabla^v\/d^v\Sigma(E,F)=\nabv E(d^v\Sigma(F))-d^v\Sigma(\nab EF),
$$
for all $E,F\in\fX(M\times I)$, where $d^v\Sigma$ is the vertical derivative; cf.\ equation \eqref{eq:vertderiv}.  This is a $\Sigma\inv\VV$-valued $2$-tensor on $M\times I$, but unlike the second fundamental form of a mapping is typically not symmetric; its asymmetry is measured by the exterior derivative: 
\begin{align}
d^v d^v\Sigma(E,F)
&=\nabla^v d^v\Sigma(E,F)-\nabla^v d^v\Sigma(F,E).
\label{eq:dvdv}
\end{align} 
If $\vartheta^v$ is viewed as a $\VV$-valued 1-form on $P$ then:
$$
d^v\Sigma=\vartheta^v\circ d\/\Sigma=\Sigma^*\vartheta^v,
$$
hence:
\begin{equation}
d^v d^v\Sigma=\Sigma^* d^v\vartheta^v. 
\label{eq:1}
\end{equation}

\begin{defn}\label{defn:curvform}
The \textsl{curvature form} of $\pi$ is the $\VV$-valued $2$-form $\Theta=d^v\vartheta^v$ on $P$, whose pullback to a $\s\inv\VV$-valued $2$-form $\s^*\Theta$ on $M$ is the \textsl{curvature} of $\s$.  The submersion $\pi$ is \textsl{flat} if $\Theta=0$, and the section $\s$ is \textsl{flat} if $\s^*\Theta=0$.   
It will sometimes be convenient to convert the curvature form $\Theta$ into a $3$-tensor $\theta$ on $P$, as follows:
\begin{equation}
\theta(A,B)C=k(A,\Theta(B,C)),
\label{eq:curvtensor}
\end{equation}
for all $A,B,C\in\fX(P)$.
\end{defn} 

\begin{rems}
\label{rems:curvform}
\item{}
\begin{itemize}[leftmargin=1.4em]
\item[1)]
If $\pi$ is flat then clearly so are all its sections; for an interesting converse see
Theorem \ref{thm:3dflat}.

\item[2)]
Definition \ref{defn:curvform} agrees with standard terminology in the familiar setting of a vector bundle with linear connection  (Remarks \ref{rems:can}\,(2) and \ref{rem:flat}).

\item[3)]
The canonical example of a flat submersion is a Riemannian product, projected onto either factor.  

\item[4)]
If $\pi$ is a Riemannian submersion then $\pi$ is flat if and only if $\pi$ is totally geodesic (Proposition \ref{prop:curvformgeom} and Lemma \ref{lem:fundform}).

\item[5)]
Evaluating equation \eqref{eq:1} at $t=0$ shows that a section $\s$ is flat precisely when its vertical second fundamental form $\nabla^v d^v\s$ is symmetric; see also Remarks \ref{rems:vbvsff}\,(2).  Thus flatness of sections generalises horizontality.

\item[6)]
Further ramifications of flatness, both of a submersion and its sections, appear in Theorems \ref{thm:firstvar1} and \ref{thm:flatharm}. 
\end{itemize}
\end{rems}

For any vector field $X$ on $M$ let $\bar X$ denote the natural extension to $M\times I$: 
$$
\bar X(x,t)=di_t(X(x)),
$$
where $i_t$ is the inclusion: 
$$
i_t\colon M\to M\times I;\,x\mapsto(x,t).
$$  
Then:
\begin{equation}
\nab{\del_t}\bar X=0=\nab{\bar{X}}\del_t.
\label{eq:2}
\end{equation}
The following result is an interim but nonetheless useful expression for the first variation of higher-power vertical energy.

\begin{thm}\label{thm:firstvar1}
For all integers $r=1,\dots,m$:
$$
\frac{d}{dt}\Big|_{t=0}\,\Evr(\s_t)
=-\int_M \bigl(k(\tvr(\s),V)
+\tr{\newt}\,\theta(d\s,d\s)V\bigr)\volume(g),
$$
where $\newt=\nu_{r-1}^v(\s)$ and $\theta$ is the curvature $3$-tensor defined in \eqref{eq:curvtensor}.  Thus if $\pi$ is flat then $\s$ is a $r$-harmonic section precisely when $\tvr(\s)=0$.
\end{thm}

\begin{proof}
We first compute the variation tensor $\b_t$ of $\a_t$, as follows:  
\begin{align*}
g\bigl(\b_t(X),Y\bigr)
&=\del_t.k\bigl(d^v\Sigma(\bar X),d^v\Sigma(\bar Y)\bigr)  \\[1ex]
&=k\bigr(\nabla^v d^v\Sigma(\del_t,\bar X),d^v \s_t(Y)\bigr)
+k\bigl(d^v \s_t(X),\nabla^v d^v\Sigma(\del_t,\bar Y)\bigr),
\quad\text{by \eqref{eq:2}} \\[1ex]
&=k\bigr(\nabla^v d^v\Sigma(\bar X,\del_t)
+d^v d^v\Sigma(\del_t,\bar X),
\,d^v \s_t(Y)\bigr) \\[0.5ex]
&\qquad\quad
+k\bigl(d^v \s_t(X),\nabla^v d^v\Sigma(\bar Y,\del_t)
+d^v d^v\Sigma(\del_t,\bar Y)\bigr),
\quad\text{by \eqref{eq:dvdv}} \\[1ex]
&=k\bigl(\nabv X V_t+\Theta(V_t,d\s_t(X)),d^v\s_t(Y)\bigr) \\[0.5ex]
&\qquad\quad
+k\bigl(d^v\s_t(X),\nabv Y V_t+\Theta(V_t,d\s_t(Y))\bigr),
\quad\text{by \eqref{eq:1} and \eqref{eq:2}.}
\end{align*}
Then by Lemma \ref{lem:eldiff}:
\begin{align*}
\frac{d}{dt}\Big|_{t=0} \vrs_r^v(\s_t)
&=\trace(\b\circ \newt) 
=\ts\sum_i g\bigl(\b(\newt e_i),e_i\bigr) 
\\[1ex]
&=\ts\sum_i k\bigl(\nabv{e_i}V
+\Theta(V,d\s(\newt e_i)),\,d^v \s(e_i)\bigr) \\[0.5ex]
&\qquad\quad
+\ts\sum_i k\bigl(d^v\s(\newt e_i),\nabv{e_i} V
+\Theta(V,d\s(e_i))\bigr) \\[1ex]
&=2k\bigl(d^v V,d^v \s\circ \newt\bigr)
-2\tr{\newt}\,\theta(d\s,d\s)V,
\quad\text{by \eqref{eq:curvtensor}.}
\end{align*}   
Finally, by Stokes' Theorem:
\begin{align*}
\int_M k(d^v V,d^v\s\circ \newt)\volume(g)
&=\int_M k\bigl(V,\d^v (d^v \s\circ \newt)\bigr)\volume(g) \\
&=-\int_M k(V,\tvr(\s))\volume(g),
\end{align*}
by Definition \ref{defn:rvtension} of higher-power vertical tension.
\end{proof}

\begin{eg}\label{eg:srhamap}
The Euler-Lagrange equations for higher-power harmonic maps $\v\colon(M,g)\to(N,h)$ (Theorem \ref{thm:rhamap}) may be obtained by applying Theorem \ref{thm:firstvar1} when $(P,k)$ is the Riemannian product $(M\times N,g\times h)$ and $\s$ is the graph of $\v$.  Since $\pi$ is flat the curvature term drops out, leaving the first variation in divergence form:
\begin{equation}
\frac{d}{dt}\Big|_{t=0}\,\E_r(\v_t)
=-\int_M h(\tau_r(\v),v)\volume(g),
\label{eq:firstvarer}
\end{equation}
where $v$ is the variation field of $\v_t$ at $t=0$.
\end{eg}

We now establish the basic geometric properties of the curvature form $\Theta$.  Let $\A$ denote the collective shape operator for the fibres of $\pi$:  
\begin{equation}
\shape HV=-\vartheta^v(\nab VH),
\label{eq:shape}
\end{equation}
for all vertical (resp.\ horizontal) vector fields $V$ (resp.\ $H$) on $P$.  The following result generalises (with a twist) a well-known characterisation of the second fundamental form of a Riemannian submersion (cf.\ Lemma \ref{lem:fundform}).

\begin{prop}\label{prop:curvformgeom}
If $V,W$ (resp.\ $H,K$) are vertical (resp.\ horizontal) vector fields  on $P$ then: 

\begin{itemize}[leftmargin=1.75em, itemsep=1ex]
\item[\rm i)]
$\Theta(V,W)=0${\rm;}

\item[\rm ii)]
$\Theta(V,H)=-\shape HV${\rm;} 

\item[\rm iii)]
$\Theta(H,K)=-\vartheta^v[H,K]$. 
\end{itemize}

\smallskip\noindent
Thus $\pi$ is flat precisely when $\pi$ has t.g.\ fibres and integrable horizontal distribution.
\end{prop}

\begin{proof}
Let $\vartheta\colon TP\to TP$ denote the identity morphism.  Then:
$$
\vartheta=\vartheta^v+\vartheta^z,
$$
hence:
$$
\Theta=d^v \vartheta^v=\vartheta^v (d\vartheta-d\vartheta^z)
=-\vartheta^v d\vartheta^z,
$$
since $d\vartheta$ is the torsion of the Levi-Civita connection of $(P,k)$.  Therefore, for all $A,B\in\fX(P)$:
\begin{align*}
\Theta(A,B)
&=-\vartheta^v\bigl(\nab A(B^z)-\nab B(A^z)-[A,B]^z\bigr) 
\notag \\
&=\vartheta^v(\nab B(A^z)-\nab A(B^z)).
\label{eq:theta}
\end{align*}
Now (i) is clear, and (iii) follows since $\nabla$ is torsion-free.  For (ii) note that:
\begin{equation*}
\Theta(V,H)=\vartheta^v(\nab VH)=-\shape HV.
\qedhere
\end{equation*}
\end{proof}
  
We extend $\A$ to a $\VV$-valued 2-tensor on $P$, by precomposing with the appropriate projections:
$$
\A(A,B)=\shape{A^z}B^v.
$$
The first variation of higher-power vertical energy may now be written in divergence form, and the Euler-Lagrange equations extracted.

\begin{thm}
\label{thm:firstvar2}
Let $\s$ be a section of a submersion $\pi\colon (P,k)\to(M,g)$ of Riemannian manifolds, and let $\s_t$ be a smooth 1-parameter variation of $\s$ through sections, with $\s_0=\s$ and variation field $V$ at $t=0$. Then for all $r=1,\dots,m${\rm:}
$$
\frac{d}{dt}\Big|_{t=0} \Evr(\s_t)
=-\int_M k(\tvr(\s)+\tr{\newt}(\s^*\A),V)\volume(g),
$$
where $\newt=\nu_{r-1}^v(\s)$.  Thus $\s$ is a $r$-harmonic section if and only if:
$$
\tvr(\s)+\tr{\newt}(\s^*\A)=0.
$$
If $\pi$ has t.g.\ fibres then $\s$ is a $r$-harmonic section precisely when $\tvr(\s)=0$. 
\end{thm}

\begin{proof}
Write $W_i=d^v\s(\newt e_i)$ and $H_i=d^z\s(e_i)$.  Then by Proposition \ref{prop:curvformgeom}\,(i),\,(ii):
\begin{align*}
\tr{\newt}\,\theta(d\s,d\s)V
&=\textstyle\sum_i k(W_i,\shape{H_i}V)
=\textstyle\sum_i k(\shape{H_i}W_i,V)  \\[0.5ex]
&=\textstyle\sum_i k\bigl(\A(d\s(e_i),d\s(\newt e_i)),V\bigr) \\[0.5ex]
&=k(\tr{\newt}(\s^*\A),V).
\end{align*}
The result now follows from Theorem \ref{thm:firstvar1}.
\end{proof}

\begin{rems}
\label{rems:firstvar2}
\item{}
\begin{itemize}[leftmargin=1.4em]
\item[1)]
If $\newt=\nu_{r-1}^v(\s)$ we have the following identity (a generalisation of Lemma \ref{lem:trace} in this situation):
$$
k(d^v\s\circ \newt (X),d^v\s(Y))
=k(\iota_X(d^v\s)^r,\iota_Y(d^v\s)^r),
$$
where the interior products on the right hand side are $(r-1)$-forms on $M$ with values in $\wedge^{r-1}(\s\inv\VV)$.  Thus $\s$ is $r$-horizontal if and only if $d^v\s\circ \newt=0$, in which case both pieces of the Euler-Lagrange operator vanish, corroborating the fact that $r$-horizontal sections are $r$-harmonic sections.

\item[2)]
In most applications, $\pi$ is a fibre bundle with connection and $k$ is a `Kaluza-Klein' metric, in which case $\pi$ is a Riemanannian submersion with t.g.\  fibres \cite{V} (see also Remark \ref{rems:can}\,(3)).
\end{itemize}
\end{rems}

It is interesting to compare the notion of higher-power harmonic sections with sections that are higher-power harmonic maps.   In \cite{Sk} it was proposed to use the $2$-nd energy (of a map) as a  perturbation of the standard ($1$-st) energy, and in this spirit we make the following:

\begin{defn}\label{defn:twistskyr}
A section $\s$ is a \textsl{twisted $r$-skyrmion} with  \textsl{coupling constants} $c_1,\dots,c_r\in\R$, $c_i\geqs0$, $c_1,c_r\neq0$, if $\s$ is a critical point with respect to variations through sections of the hybrid functional:
$$
c_1\/\E^v_{\/1}(\s)+c_2\/\E_{\/2}^v(\s)+\cdots+c_r\/\Evr(\s).
$$
\end{defn}

\begin{thm}\label{thm:hasec}
Let  $\s$ be a section of a Riemannian submersion with t.g.\ fibres.  Then for all $r=1,\dots,m$:
$$
\vrs_r(\s)
=\vrs_r^v(\s)+(m-r+1)\vrs_{r-1}^v(\s)+\cdots
+\textstyle\binom{m-1}{r-1}\vrs_1^v(\s)+\textstyle\binom{m}{r},
$$
and for all $r=1,\dots,m-1${\rm:}
$$
\nu_r(\s)
=\nu_r^v(\s)+(m-r)\nu_{r-1}^v(\s)+\cdots
+\textstyle\binom{m-2}{r-1}\nu_1^v(\s)
+\textstyle\binom{m-1}{r}1.
$$
Furthermore the vertical component of $\tau_r(\s)$ is:
$$
\tau_r(\s)^v=\tvr(\s)+(m-r+1)\tau^v_{r-1}(\s)
+\cdots+\textstyle\binom{m-1}{r-1}\tau^v_1(\s).
$$
Thus $\s$ is a $r$-harmonic map precisely when $\s$ is a twisted $r$-skyrmion with coupling constants $c_i=\binom{m-i}{r-i}$
and the horizontal component of $\tau_r(\s)$ vanishes.
\end{thm}

\begin{proof}
For clarity in this context, let $\a$ (resp.\ $\a^v$) denote the Cauchy-Green (resp.\ vertical Cauchy-Green) tensor of $\s$.  Since $\pi$ is a Riemannian submersion:
\begin{align*}
g(\a(X),Y)
&=k(d\s(X),d\s(Y)) \\
&=k(d^v\s(X),d^v\s(Y))+k(d^z\s(X),d^z\s(Y)) \\
&=g(\a^v(X),Y)+g(X,Y).
\end{align*}
Thus $\a=1+\a^v$, and the expressions for $\vrs_r(\s)$ and $\nu_r(\s)$ follow from Lemma \ref{lem:char}.  Then $\tau_r(\s)^v$ follows by expanding $\E_r(\s)$ as a sum of higher-power vertical energies and comparing the first variations for each using \eqref{eq:firstvarer} and Theorem \ref{thm:firstvar2}.
\end{proof}

When $r=1$ Theorem \ref{thm:hasec} simplifies to:
$$
\tau(\s)^v=\tau^v(\s);
$$
thus $\s$ is a harmonic map precisely when $\s$ is a harmonic section and the horizontal component $\tau(\s)^z$ vanishes.  This is familiar, for example, in the theory of `harmonic unit vector fields', in which context $\tau(\s)^z$ may be expressed as a `twisted Ricci curvature' \cite{Gil}.  In fact this (second) link with curvature generalises fairly comprehensively (Theorem \ref{thm:flatharm}).  To analyse $\tau_r(\s)^z$ we will need the following well-known decomposition of the second fundamental form of a Riemannian submersion, compiled from results of \cite{H,ON,V}.

\begin{lem}
\label{lem:fundform}
Suppose $\pi\colon (P,k)\to(M,g)$ is a Riemannian submersion.
If $H,K$ (resp.\ $V,W$) are horizontal (resp.\ vertical) vector fields on $P$ then:

\begin{itemize}[leftmargin=1.8em, itemsep=1ex]
\item[\rm i)]
$\nabla d\pi(H,K)=0${\rm;}

\item[\rm ii)]
$\nabla d\pi(V,W)=-d\pi(\nab VW)${\rm;}  

\item[\rm iii)]
$g(\nabla d\pi(V,H),d\pi(K))=\tfrac12\/k(V,[\/H,K\/])$.
\end{itemize}
In particular, $\pi$ is a t.g.\ map if and only if $\pi$ is flat.
\end{lem}

\begin{lem}
\label{lem:flatharm}
Let $\s$ be a section of a Riemannian submersion $\pi\colon (P,k)\to(M,g)$ with t.g.\ fibres.  Then for all $X,Y,Z\in\fX(M)${\rm:}
$$
2\/g(d\pi\circ\nabla d\s(X,Y),Z)
=\s^*\theta(X,Y)Z+\s^*\theta(Y,X)Z,
$$
where $\theta$ is the curvature $3$-tensor defined in \eqref{eq:curvtensor}.
\end{lem}

\begin{proof}
Successive differentiation of the equation $\pi\circ\s=1_M$ yields:
$$
\nabla d\pi(d\s,d\s)+d\pi\circ\nabla d\s=0.
$$
Applying first Lemma \ref{lem:fundform}:
\begin{align*}
g(d\pi\circ\nabla d\s(X,X),Z)
&=-g\bigl(\nabla d\pi(d\s(X),d\s(X)),Z\bigr) \\
&=-2\/g\bigl(\nabla d\pi(d^v\s(X),d^z\s(X)),Z\bigr) \\
&=-k(d^v\s(X),\,[\/d^z\s(X),\/d^z\s(Z)\/]),
\shortintertext{and then Proposition \ref{prop:curvformgeom}:} 
&=k\bigl(d^v\s(X),\Theta(d^z\s(X),d^z\s(Z))\bigr) \\
&=k\bigl(d\s(X),\Theta(d\s(X),d\s(Z))\bigr) \\
&=\s^*\theta(X,X)Z,
\quad\text{by \eqref{eq:curvtensor}.}
\end{align*}
The identity follows by polarisation.
\end{proof}

Applying Lemma \ref{lem:flatharm} to Corollary \ref{cor:srtension}, in conjunction with Theorem \ref{thm:hasec}, allows us to characterise the harmonicity of flat sections.

\begin{thm}
\label{thm:flatharm}
If $\s$ is a section of a Riemannian submersion $\pi\colon (P,k)\to(M,g)$ with t.g.\ fibres then the horizontal component of $\tau_r(\s)$ is given by:
$$
g(d\pi\circ\tau_r(\s),X)
= g(\diverge \newt,X) + \tr{\newt}\/\s^*\theta(\cdot,\cdot)X,
$$
for all $X\in \fX(M)$, where $\newt=\nu_{r-1}(\s)$.  In particular, if $\s$ is flat then:
$$
d\pi\circ\tau_r(\s)=\diverge\newt,
$$ 
so $\s$ is a $r$-harmonic map precisely when $\s$ is a twisted $r$-skyrmion with coupling constants $c_i=\binom{m-i}{r-i}$ and the full Newton tensor $\nu_{r-1}(\s)$ is solenoidal.
\end{thm}

Since horizontal sections are flat (Remarks \ref{rems:curvform}) we obtain the following synopsis of the horizontal case.

\begin{thm}\label{thm:harmhoriz}
Let $\s$ be a section of a submersion $\pi$ of Riemannian manifolds.

\begin{itemize}[leftmargin=1.5em, itemsep=0.5ex]
\item[\rm i)]
If $\s$ is $r$-horizontal then $\s$ is a $r$-harmonic section.

\item[\rm ii)]
If $\pi$ is a Riemannian submersion with t.g.\ fibres and $\s$ is horizontal then $\s$ is a $r$-harmonic map for all $r=1,\dots,m$.
\end{itemize}
\end{thm}

\begin{proof} 
Part (i) follows from Remark \ref{rems:firstvar2}\,(1).  
Part (ii) follows from Theorem \ref{thm:flatharm}.  For, if $\s$ is horizontal then $\s$ is a $r$-harmonic section for all $r$, hence a twisted $r$-skyrmion for any coupling.  Furthermore, all vertical Newton tensors of $\s$ vanish, so the full Newton tensor $\nu_{r-1}(\s)$ is a constant multiple of the identity (Theorem  \ref{thm:hasec}), hence solenoidal.  
\end{proof}

\begin{rem}
In certain situations Theorem \ref{thm:harmhoriz} becomes rigid; see for example Theorem \ref{thm:compactvb} and Corollary \ref{cor:hphasecsb}.
\end{rem}

\section{Higher-power harmonic sections of Riemannian vector bundles}
\label{sec:vbundles}

We interpret the results of Section \ref{sec:firstvar} when the submersion $\pi$ is a \textsl{Riemannian vector bundle;} that is, a vector bundle $\pi\colon\EE\to M$ equipped with a linear connection $\nabla$ and holonomy-invariant fibre metric $\<\cdot,\cdot\>$.  Our analysis makes heavy use of the associated connection map, which we briefly review. The \textsl{square} of $\EE$ \cite{St} is the pullback bundle $\pi\inv\EE\to\EE$:
$$
\pi\inv\EE=\{(v,w)\in\EE\times\EE:\pi(v)=\pi(w)\},
$$
which may be equipped with the pullbacks of both $\nabla$ and the fibre metric.  The \textsl{connection map} for $\nabla$ is the $\pi\inv\EE$-valued 1-form $\k$ on $\EE$ defined for all $A\in T\EE$ by:
\begin{equation}
\k(A)=\nab A\upchi,
\label{eq:connmap}
\end{equation}
where $\upchi$ is the \textsl{diagonal section} of the square bundle, defined for all $v\in\EE$ by:
$$
\upchi(v)=(v,v).
$$
Then $\k$ is surjective and $\k|_\VV=\iota$, where $\iota\colon\VV\to\pi\inv\EE$ is the \textsl{canonical isomorphism,} obtained by amalgamating the canonical identifications of vertical tangent spaces with the fibre of $\EE$ to which they are tangent.  Given a Riemannian metric $g$ on $M$, the Riemannian metric of choice on $\EE$ is the \textsl{Sasaki metric,} defined:
\begin{equation}
k(A,B)=g(d\pi(A),d\pi(B))+\<\k(A),\k(B)\>,
\label{eq:sasaki}
\end{equation}
for all $A,B\in\fX(\EE)$.
Then $\pi$ is a Riemannian submersion with $\HH=\ker\k$, the horizontal distribution of $\nabla$.  The \textsl{curvature form} $K$ of $\nabla$ is the exterior covariant derivative of $\k$:
\begin{align}
K(A,B)
=d^\nabla\!\k(A,B)
&=\nab A(\k B)-\nab B(\k A)-\k[A,B] 
\notag \\[0.5ex] 
&=R^\nabla\!(A,B)\upchi,
\label{eq:curv}
\end{align}
where $R^\nabla$ is the curvature tensor of $\nabla$.  Thus $K$ is a $\pi\inv\EE$-valued 2-form on $\EE$; we note two of its properties.

\begin{prop}\label{prop:curv}
\item{}
\begin{itemize}[leftmargin=1.5em, itemsep=0.5ex]
\item[\rm i)] 
$K$ is horizontal; that is, $K(A,B)=0$ whenever $A$ or $B$ is vertical.

\item[\rm ii)]
$K$ measures the failure of $\k$ to intertwine the Levi-Civita connection of the Sasaki metric with the linear connection in $\EE$:
$$
\nab A(\k B)-\k(\nab AB)=\tfrac12\/K(A,B).
$$
\end{itemize}
\end{prop}

\begin{proof}
(i) is standard, and (ii) follows, after some computation, from the Koszul characterisation of the Levi-Civita connection (in which the holonomy-invariance of $\<\cdot,\cdot\>$ is crucial).  
\end{proof}

\begin{rems}
\label{rems:can}
\item{}
\begin{itemize}[leftmargin=1.4em]
\item[1)]
Since $\k|_\VV=\iota$ it follows from Proposition \ref{prop:curv} that the canonical isomorphism is connection-preserving:
\begin{equation}
\nab A(\iota V)=\iota(\nabv A V),
\label{eq:can}
\end{equation}
for all vertical vector fields $V$ on $\EE$.  This fact is crucial.

\item[2)]
Proposition \ref{prop:curv} confirms that Definition \ref{defn:curvform} of the curvature form $\Theta$ is consistent with established terminology:
\begin{align}
\iota\circ\Theta(A,B)
&=\iota\circ d^v \vartheta^v(A,B) 
\notag \\
&=\iota\,\nabv A(\vartheta^vB)-\iota\,\nabv B(\vartheta^vA)-\iota\circ\vartheta^v[A,B] 
\notag \\
&=\nab A(\k B)-\nab B(\k A)-\k[A,B],
\quad\text{by \eqref{eq:can}}
\notag \\
&=K(A,B).
\label{eq:curvform}
\end{align}

\item[3)]
Placing \eqref{eq:curvform} alongside Propositions \ref{prop:curv} and \ref{prop:curvformgeom} confirms that $\pi$ has t.g.\ fibres (see also Remark \ref{rems:firstvar2}).
\end{itemize}
\end{rems}  
 
Now let $\s$ be a section of $\pi$.  Equation \eqref{eq:connmap} yields the characteristic property:
\begin{equation}
\iota\circ d^v\s(X)=\k(d\s(X))=\nab X\s,
\label{eq:cov}
\end{equation}
for all $X\in TM$, after the natural identification of $\s\inv\pi\inv\EE$ with $\EE$.  Thus $\s$ is horizontal precisely when $\s$ is parallel.

\begin{rem}
\label{rem:flat}
From \eqref{eq:curv} and \eqref{eq:curvform}:
\begin{equation}
\iota\circ\s^*\Theta(X,Y)=\Rnabla(X,Y)\s.
\label{eq:vbcurvtensor}
\end{equation}
Thus $\s$ is flat (Definition \ref{defn:curvform}) precisely when $\Rnabla(X,Y)\s=0$ for all $X,Y\in\fX(M)$. 
\end{rem}

Plugging \eqref{eq:cov} into \eqref{eq:vcg} yields the vertical Cauchy-Green tensor:
\begin{equation}
g(\a(X),Y)
=\<\nab X\s,\nab Y\s\>,
\label{eq:vcgvb}
\end{equation}
from which it follows that:
\begin{equation}
\vrs_r^v(\s)=\|(\nabla\s)^r\|^2,
\label{eq:srvbdensity}
\end{equation}
where:
$$
(\nabla\s)^r=\nabla\s\wedge\cdots\wedge\nabla\s,
$$
the $r$-fold exterior product. 

\begin{defn}\label{defn:rpar}
A section $\s$ is \textsl{$r$-parallel} if the rank of $\nabla\s$ is strictly less than $r$.  Thus $\s$ is $1$-parallel if and only if $\s$ is parallel in the usual sense, and by \eqref{eq:cov} $\s$ is $r$-parallel precisely when $\s$ is $r$-horizontal (Definition \ref{defn:rhoriz}). 
\end{defn}

\par
When $M$ is compact Theorem \ref{thm:harmhoriz} becomes rigid.

\begin{thm}\label{thm:compactvb}
Let $\s$ be a section of a Riemannian vector bundle with compact 
base, equipped with the Sasaki metric.  For all $r=1,\dots,m${\rm:}

\begin{itemize}[leftmargin=1.5em, itemsep=0.5ex]
\item[\rm i)]
$\s$ is a $r$-harmonic section if and only if $\s$ is $r$-parallel.

\item[\rm ii)]
$\s$ is a $r$-harmonic map if and only if $\s$ is parallel.
\end{itemize}
\end{thm}

\begin{proof}
Consider the variation:
$$
\s_t=(1+t)\s,
\quad t>-1.
$$ 
Then by \eqref{eq:srvbdensity}:
$$
\vrs_r^v(\s_t)=(1+t)^{2r}\|(\nabla\s)^r\|^2,
$$
hence by \eqref{eq:rvertenergy}:
$$
\frac{d}{dt}\Big|_{t=0} \Evr (\s_t)
=r\int_M\|(\nabla\s)^r\|^2\volume(g),
$$
which yields (i).  By Theorem \ref{thm:hasec}:
\begin{align*}
\frac{d}{dt}\Big|_{t=0} \E_r(\s_t)
&=2\int_M \bigl(r\|(\nabla\s)^r\|^2
+\cdots+\textstyle\binom{m-1}{r-1}\|\nabla\s\|^2\bigr)\volume(g),
\end{align*}
hence $\s$ is a $r$-harmonic map precisely when $\nabla\s=0$. 
\end{proof}

\begin{rems}
\item{}
\begin{itemize}[leftmargin=1.4em]
\item[1)]
Theorem \ref{thm:compactvb} is well-known when $r=1$: if $M$ is compact every harmonic section of $\EE$ is parallel \cite{I, N, W1}.

\item[2)]
There is a generalisation to the non-compact environment for sections of constant length (Corollary \ref{cor:hphasecsb}). 
\end{itemize}
\end{rems}
 
For a general characterisation of $r$-harmonic sections, we use the canonical isomorphism to realise the higher-power vertical tension fields as sections of $\pi$:
\begin{equation}
\T_r(\s)=\iota\circ\tvr(\s),
\label{eq:srtensionvb}
\end{equation}
and recall the second covariant derivative:
\begin{equation}
\nabsq XY\s=\nab X(\nab Y\s)-\nab{\nab XY}\s,
\label{eq:2cov}
\end{equation}
which entwines the linear connection in $\pi$ and the Levi-Civita connection of $(M,g)$.

\begin{thm}\label{thm:hphasecvb}
Let $\s$ be a section of a Riemannian vector bundle $\EE$, equipped with the Sasaki metric.  Then:
$$
\T_r(\s)=\tr{\newt} \nabla^2\s+\nab{\diverge \newt}\s,
$$
where $\newt=\nu_{r-1}^v(\s)$, and $\s$ is a $r$-harmonic section if and only if $\T_r(\s)=0$.
\end{thm}

\begin{proof}
Applying the canonical isomorphism to Theorem \ref{thm:vsrtension}, and using \eqref{eq:cov}:
\begin{equation*}
\T_r(\s)=\nab{\diverge \newt}\s+\iota(\tr{\newt} \nabla^v d^v\s).
\label{eq:45}
\end{equation*}
By Proposition \ref{prop:curv}:
\begin{align}
\iota\,\nabla^v d^v\s(X,Y)
&=\iota\,\nabv X(d^v\s(Y))
-\iota\circ d^v\s(\nab XY)
\notag \\
&=\nab X(\nab Y\s)-\nabla\s(\nab XY),
\quad\text{by \eqref{eq:can} and \eqref{eq:cov}}
\notag \\
&=\nabsq XY\s.
\label{eq:vsff}
\end{align}
Hence:
$$
\iota(\tr{\newt} \nabla^v d^v\s)=\tr{\newt} \nabla^2\s.
$$
The characterisation of $r$-harmonic sections follows from Theorem \ref{thm:firstvar2}, since $\pi$ has t.g.\ fibres.
\end{proof}

\begin{rems}\label{rems:vbvsff}
\item{}
\begin{itemize}[leftmargin=1.4em]
\item[1)]
If $r=1$ then Theorem \ref{thm:hphasecvb} yields the familiar characterisation $\nabla^*\nabla\s=0$ for harmonic sections $\s$, where $\nabla^*\nabla=-\trace\nabla^2$ is the \textsl{rough Laplacian.}   

\item[2)]
From equation \eqref{eq:vsff} in the proof of Theorem \ref{thm:hphasecvb}, the curvature $R^\nabla(\cdot,\cdot)\s$ is the antisymmetrisation of the vertical second fundamental form of $\s$;  thus the latter is symmetric if and only if $\s$ is flat (see also Remarks \ref{rems:curvform}).
\end{itemize}
\end{rems}

\begin{cor}\label{cor:hphasecvb}
With the same hypotheses as Theorem \ref{thm:hphasecvb}, if $\s$ is a $r$-harmonic section of $\EE$ then so is $c\/\s$ for all $c\in\R$.
\end{cor}

\begin{proof}
If $\hat\s=c\/\s$ then $\hat\a=c^2\a$ by \eqref{eq:vcgvb}, and by the homogeneity of the vertical Newton tensor (Lemma \ref{lem:char}):
$$
\hat\newt=\nu_{r-1}^v(\hat\s)=\upchi_{r-1}(\hat\a)
=c^{2r-2}\,\upchi_{r-1}(\a)=c^{2r-2}\newt.
$$
Therefore by Theorem \ref{thm:hphasecvb}:
\begin{equation}
\T_r(\hat\s)=c^{2r-1}\,\T_r(\s),
\label{eq:trscaled}
\end{equation}
from which the result follows.
\end{proof}

\begin{rem}\label{rem:rpowerspace}
The $r$-harmonic sections of $\EE$ do not generally constitute a linear subspace if $r>1$; see however Theorem \ref{thm:3dzero} and Remark \ref{rems:3dzero}\,(2).
\end{rem}

We also interpret the expression for the horizontal component of $\tau_r(\s)$ from Theorem \ref{thm:flatharm}.  For any $\EE$-valued $1$-form $\eta$ on $M$ we define the Ricci-type $\EE$-valued $1$-form $S_\eta$ by:
\begin{equation}
S_\eta(X)=\ts\sum_i \Rnabla (X,e_i)\eta(e_i).
\label{eq:Seta}
\end{equation}

\begin{thm}\label{thm:taurhoriz}
The horizontal component of $\tau_r(\s)$ is given by:
$$
g(d\pi\circ\tau_r(\s),X)
= g(\diverge \newt,X) + \<S_{\eta}(X),\s\>,
$$
for all $X\in \fX(M)$, where $\newt=\nu_{r-1}(\s)$ and $\eta=(\nabla\s)\circ \newt$.  In particular, if $\EE=TM$ equipped with the standard Riemannian structure then:
$$
d\pi\circ\tau_r(\s)=\diverge \newt+\ts\sum_i R(\s,\eta(e_i))e_i,
$$
where $R$ is the Riemann tensor.
\end{thm}

\begin{proof}
By equations \eqref{eq:cov} and \eqref{eq:vbcurvtensor} and the holonomy-invariance of $\nabla$:
$$
\s^*\theta(Y,Z)X
=\<\nab Y\s,\Rnabla (Z,X)\s\>
=\<\Rnabla (X,Z)\nab Y\s,\s\>.
$$
Therefore since $\newt$ is self-adjoint:
$$
\tr{\newt}\,\s^*\theta(\cdot,\cdot)X
=\ts\sum_i \<\Rnabla (X,e_i)\nab{\newt e_i}\s,\s\>
=\<S_{\eta}(X),\s\>.
$$
The result now follows from Theorem \ref{thm:flatharm} and  additional symmetries of the Riemann tensor.
\end{proof}

The rigidity of Theorem \ref{thm:compactvb} may be mitagated for sections of constant length $q>0$ (topology permitting) by restricting the entire variational problem to the sphere subbundle:
$$
\SS\EE(q)=\{v\in\EE:\<v,v\>=q^2\}.
$$   
The metric on $\SS=\SS\EE(q)$ is simply the restriction of the Sasaki metric.  It follows from holonomy-invariance of the fibre metric of $\EE$ that $\SS$ is a holonomy-invariant subbundle; hence the horizontal distribution of $\EE$ is tangent to $\SS$, and therefore coincides with the horizontal distribution of $\SS$.  This simplifies things considerably; for example, the vertical derivative of a section $\s$ of $\SS$ is unchanged if $\s$ is regarded as a section of $\EE$, and consequently so are its higher-power vertical energies and vertical Newton tensors.  There will however be a change in the higher-power vertical tension fields, which we analyse via the first variation.  

\begin{thm}\label{thm:hphasecsb}
Suppose $\s$ is a section of $\EE$ with constant length $q>0$.  Then $\s$ is a $r$-harmonic section of the sphere bundle $\SS\EE(q)$ if and only if $\T_r(\s)$ is a pointwise multiple of $\s$.  The Euler-Lagrange equations are:
$$
\T_r(\s)=-\frac{r}{q^2}\,\|(\nabla\s)^r\|^2\s,
$$
where $\T_r(\s)$ is given by Theorem \ref{thm:hphasecvb}.
\end{thm}

\begin{proof}
By Theorem \ref{thm:firstvar2} and definition \eqref{eq:sasaki} of the Sasaki metric:
\begin{equation}
\frac{d}{dt}\Big|_{t=0} \Evr(\s_t)
=-\int_M \<\T_r(\s),\iota V\>\volume(g).
\label{eq:firstvarvb}
\end{equation}
Now $\iota V$ is a section $\z$ of $\EE$, which since $\s_t$ is a variation through sections of constant length satisfies:
$$
\<\z,\s\>
=\frac{d}{dt}\Big|_{t=0}\<\s_t,\s\>
=\frac12 \frac{d}{dt}\Big|_{t=0}\<\s_t,\s_t\>=0.
$$
Conversely, if $\z$ is a section of $\EE$ pointwise orthogonal to $\s$ then it is possible to construct a variation of $\s$ in $\SS$ with variation field $\iota V=\z$, for example by appropriately rescaling $\s+t\z$.  It follows that $\s$ is a $r$-harmonic section of $\SS$ if and only if $\T_r(\s)=f\s$ for some smooth function $f\colon M\to\R$.  Since $\s$ has constant length and the fibre metric is holonomy-invariant:
$$
\<\nab X\s,\s\>=0,
\qquad
\<\nabsq XY\s,\s\>=-\<\nab X\s,\nab Y\s\>=-\<\a(X),Y\>,
$$
by \eqref{eq:vcgvb}.
Therefore by Theorem \ref{thm:hphasecvb}:
\begin{align*}
q^2f
=\<\T_r(\s),\s\>
&=\<\tr{\newt}\nabla^2\s,\s\>
=-\trace(\a\circ \newt) \\
&=-r\/\vrs_r^v(\s)
=-r\/\|(\nabla\s)^r\|^2,
\end{align*}
by Lemma \ref{lem:trace} and equation \eqref{eq:srvbdensity}.
\end{proof}

\begin{rems}
\item{}
\begin{itemize}[leftmargin=1.4em]
\item[1)]
When $r=q=1$ the Euler-Lagrange equations of Theorem \ref{thm:hphasecsb} reduce to:
\begin{equation}
\nabla^*\nabla\s=\lVert\nabla\s\rVert^2\s,
\label{eq:s1hasecsb}
\end{equation}
familiar from \cite{Wie, W3} and subsequent papers on `harmonic unit vector fields' \cite{Gil}.  

\item[2)]
Theorem \ref{thm:hphasecsb} shows in effect that the higher-power vertical tension fields of a section $\s$ of $\SS$ are obtained by orthogonally projecting onto $T\SS$ those of $\s$ when regarded as a section of $\EE$, as we would expect.
\end{itemize}
\end{rems}

\begin{cor}
\label{cor:hphasecsb}
Suppose $\s$ is a section of $\EE$ with constant length.  Then $\s$ is a $r$-harmonic section of $\EE$ if and only if $\s$ is $r$-parallel.  In particular, $\s$ is a $r$-harmonic section of $\EE$ for all $r\geqs\rank\EE$. 
\end{cor}

\begin{proof}
If $\s$ is a non-zero $r$-harmonic section of $\EE$ then it follows from Theorems \ref{thm:hphasecvb} and \ref{thm:hphasecsb} that $\s$ is a $r$-harmonic section of $\SS$ with $(\nabla\s)^r=0$.  Furthermore, since $\s$ has constant length and the fibre metric is holonomy-invariant, $\nabla\s$ takes values in the corank $1$ subbundle $\uperp{\s}\subset\EE$ and therefore has rank strictly less than $p-m$; so $\s$ is $r$-parallel for all $r\geqs p-m$.
\end{proof}

\begin{cor}
\label{cor:hphasecusb}
Let $\s$ be a section of $\EE$ with constant length $q>0$.  Then $\s$ is a $r$-harmonic section of $\SS\EE(q)$ if and only if $\s/q$ is a $r$-harmonic section of $\SS\EE(1)$.
\end{cor}

\begin{proof}
If $\hat\s=\s/q$ it follows from equation \eqref{eq:trscaled} in the proof of Corollary \ref{cor:hphasecvb} that $\T_r(\s)$ is a multiple of $\s$ if and only if $\T_r(\hat\s)$ is a multiple of $\hat\s$.  
\end{proof}

The following result generalises the well-known characterisation of harmonic maps into spheres \cite{Sm}:

\begin{cor}
\label{cor:hphamapsphere}
A mapping $\v\colon (M,g)\to S^n$ is $r$-harmonic precisely when:
$$
\tr{\newt}\/H_\v+d\v(\diverge\newt)+r\/\norm{(d\v)^r}^2\/\v=0,
$$
where $H_\v$ is the Hessian of $\v$ viewed as a map $M\to\R^{n+1}$ and $\newt=\nu_{r-1}(\v)$.
\end{cor}

\begin{note*}
From Corollary \ref{cor:hphasecsb}, a mapping $\v\colon(M,g)\to\R^{n+1}$ of constant length is $r$-harmonic if and only if $\rank\v<r$.
\end{note*}

\section{Higher-power harmonic vector fields on 3-dimensional Lie groups}
\label{sec:3dlie}

Suppose now that $M$ is a 3-dimensional Lie group, henceforward denoted $G$, which for simplicity we assume to be unimodular.  Let $g$ be a left-invariant Riemannian metric, and $\EE=TG$ with the standard Riemannian structure; ie.\ $\<\cdot,\cdot\>=g$ and $\nabla$ is the Levi-Civita connection.  We refer to a $r$-harmonic section of $\EE$ as a \textsl{$r$-harmonic vector field.}  We will restrict attention to \textsl{invariant} (ie.\ left-invariant) vector fields $\s$, which therefore have constant length, and hence may also be regarded as sections of $\SS\EE(q)$ where $q=\abs{\s}$.  By Corollary \ref{cor:hphasecusb} it suffices to consider $q=1$, and therefore confine attention to $r$-harmonic sections of the unit tangent bundle $UG$; we refer to these as \textsl{$r$-harmonic unit vector fields.}

\par
We briefly review the geometry and algebraic structure of $(G,g)$, following \cite{Mil} (see also \cite{MP}).  A choice of orientation determines a unique `cross product' $\times$ on the Lie algebra $\fg$; the \textsl{Lie structure map} $L\colon \fg\to\fg$ is then the unique linear map satisfying: 
\begin{equation}
L(\v\times\psi)=[\v,\psi],
\label{eq:strmap}
\end{equation}
for all $\v,\psi\in\fg$.  Then $L$ is self-adjoint precisely when $G$ is unimodular; a positively-oriented orthonormal eigenbasis $(\s_1,\s_2,\s_3)$ of $\fg$ with $L(\s_i)=\lam_i\s_i$ therefore satisfies the commutation relations:
\begin{equation}
[\s_i,\s_j]=\e_{ijk}\,\lam_k\/\s_k,
\label{eq:commrel}
\end{equation}
where $\e_{ijk}$ is the Levi-Civita symbol.  We refer to the eigenvalues $\lam_i$ as the \textsl{principal structure constants} of $(G,g)$, the eigenvectors $\s_i$ as \textsl{principal structure directions.}  An invariant plane field $\fp\subset\fg$ is a \textsl{principal section} if $\fp$ is spanned by principal directions; such a $\fp$ need not be an eigenspace of $L$.  Since all $2$-dimensional unimodular Lie algebras are abelian, it follows from \eqref{eq:strmap} and \eqref{eq:commrel} that all $2$-dimensional subalgebras $\fh\subset\fg$ are principal sections, characterised by the vanishing of the principal structure constant for the (principal) direction orthogonal to $\fh$.

\par
Being dependent on orientation, the principal structure constants are determined only up to sign, and their relative signs classify $\fg$ algebraically into one of:
\begin{equation}
\fs\fu(2),\quad
\fs\fl(2),\quad
\fe(2),\quad
\fe(1,1),\quad
\nil,\quad
\fa,
\label{eq:unimod6}
\end{equation}
where $\fe(2)$ (resp.\ $\fe(1,1)$) is the Lie algebra of the isometry group of the Euclidean (resp.\ Minkowski) plane, $\nil$ is the Lie algebra of the Heisenberg group and $\fa$ is the $3$-dimensional abelian Lie algebra.  

\par
The  \textsl{Milnor numbers} of $(G,g)$ are defined:
\begin{equation}
\mu_i=\tfrac12(\lam_1+\lam_2+\lam_3)-\lam_i.
\label{eq:mui}
\end{equation}
The Levi-Civita connection is then characterised on principal structure directions:
\begin{equation}
\nab{\s_i}\s_j
=\e_{ijk}\,\mu_i\/\s_k.
\label{eq:3dunimodlc}
\end{equation} 
In particular, the $\s_i$ are geodesic vector fields (although not in general Killing).
If:
$$
\s=a_1\s_1+a_2\s_2+a_3\s_3,\quad a_i\in\R,
$$ 
we define the \textsl{Milnor map} $M\colon\fg\to\fg$ by:
$$
M(\s)=\mu_1a_1\/\s_1+\mu_2\/a_2\/\s_2+\mu_3\/a_3\/\s_3,
$$
noting that $M$ is well-defined (up to orientation) since $\lam_i=\lam_j$ if and only if $\mu_i=\mu_j$.  For any positive integer $r$ we abbreviate the $r$-th iterate of $M$:
\begin{equation}
\s^{(r)}=M^r(\s)
=\mu_1^{\;r}a_1\/\s_1+\mu_2^{\;r}a_2\/\s_2+\mu_3^{\;r}a_3\/\s_3.
\label{eq:phir}
\end{equation}
It then follows from \eqref{eq:3dunimodlc} that:
\begin{equation}
\nab{\v}\s=\v^{(1)}\times\s,
\label{eq:lcunimod}
\end{equation}
for all $\v\in\fg$.

\par
The principal structure directions are also principal Ricci directions, with \textsl{principal Ricci curvatures:}
\begin{equation}
\rho_i=\Ric(\s_i,\s_i)
=2\mu_j\mu_k,
\label{eq:prinric}
\end{equation}
for $\{i,j,k\}=\{1,2,3\}$.  It follows that $\Ric$ is non-degenerate if and only if $M$ is invertible, and the Ricci kernel $\fn$ has dimension $0$, $2$ or $3$; if $2$-dimensional, $\fn$ is a principal section, although not in general a subalgebra or an eigenspace of $L$.  The \textsl{principal sectional curvatures} are:
\begin{equation}
K_{ij}=K(\s_i,\s_j)
=\tfrac12(\rho_i+\rho_j-\rho_k)
=(\mu_i+\mu_j)\mu_k-\mu_i\mu_j,
\label{eq:prinsec}
\end{equation}
and the Riemann tensor is characterised: 
\begin{equation}
R(\s_i,\s_j)\s
=-\e_{ijk}\/K_{ij}\,\s_k\times\s
=K_{ij}(a_j\/\s_i-a_i\/\s_j).
\label{eq:3driem}
\end{equation}  

\par
The following list (\textsl{`Milnor's list'}) summarises the geometric possibilities for each of the six classes of unimodular Lie algebra.  We assume that $\lam_1\geqs\lam_2\geqs\lam_3$, with no fewer $\lam_i$ positive than negative.

\smallskip
\begin{itemize}[leftmargin=3.9em, itemsep=1ex]

\item[$\fa\colon$] all structure constants vanish; all left-invariant metrics are flat. 

\item[$\nil\colon$]  $\lam_1>0$ and $\lam_2=\lam_3=0$.  Then $\rho_2=\rho_3<0$ and $\rho_1=-\rho_2$.

\item[$\fe(1,1)\colon$] $\lam_1>0$, $\lam_2=0$ and $\lam_3<0$.   Then  $\rho_1=-\rho_3$ and $\rho_2<-\abs{\rho_1}$, with $\rho_1=\rho_3=0$ precisely when $\lam_1=-\lam_3$.

\item[$\fe(2)\colon$] $\lam_1,\lam_2>0$ and $\lam_3=0$.  
Then $\rho_1=-\rho_2$ and $-\abs{\rho_1}<\rho_3<0$,
unless $\lam_1=\lam_2$ in which case the metric is flat.

\item[$\fs\fl(2)\colon$]
$\lam_1,\lam_2>0$ and $\lam_3<0$.  Then $\rho_2<0$ and $\rho_1\rho_3\leqs0$, with the $\abs{\rho_i}$ distinct unless:

\par
$\lam_1=\lam_2$ in which case $\rho_1=\rho_2$;

\par
$\lam_1=\lam_2-\lam_3$ in which case $\rho_1=\rho_3=0$.

\item[$\fs\fu(2)\colon$]
$\lam_1,\lam_2,\lam_3>0$.  Then $\rho_1>0$ and $\rho_2\rho_3\geqs0$, with the $\abs{\rho_i}$ distinct and non-zero unless:

\par
$\lam_1=\lam_2=\lam_3$ in which case $\rho_1=\rho_2=\rho_3$;

\par
$\lam_1=\lam_2>\lam_3$ in which case $\rho_1=\rho_2>\rho_3>0$;

\par
$\lam_1>\lam_2=\lam_3$ in which case $\rho_1>\rho_2=\rho_3$, with $\rho_3\neq0$ unless $\lam_1=2\lam_2$;

\par
$\lam_1=\lam_2+\lam_3$ in which case $\rho_2=\rho_3=0$.
\end{itemize}
Notably, metrics with degenerate Ricci curvature occur in every class except $\nil$, but the only non-abelian flat metric occurs when $\fg=\fe(2)$, and this metric is unique up to homothety.  We note also, from \eqref{eq:commrel}, that the derived subalgebra $[\fg,\fg]$ of $\fg=\fe(2)$ is $2$-dimensional; if $G=E(2)$, the Euclidean group, it is the Lie algebra of the translation subgroup.

\begin{thm}\label{thm:3dflat}
Suppose $\s$ is a non-zero invariant vector field on $(G,g)$.  Then:

\begin{itemize}[leftmargin=1.5em, itemsep=0.5ex]
\item[\rm i)]
$\s$ is flat if and only if $(G,g)$ is flat.

\item[\rm ii)]
$\s$ is parallel if and only if
$(G,g)$ is flat and $\s$ is orthogonal to $[\fg,\fg]$.
\end{itemize}
\end{thm}

\begin{proof}
\item{}
i)\;
If $\s\neq0$ is flat then it follows from \eqref{eq:3driem} that at least two principal sectional curvatures vanish.  Then by \eqref{eq:prinsec} two principal Ricci curvatures are equal and the third vanishes.  But $\dim\fn\neq1$ so all the $\rho_i$ vanish, rendering $(G,g)$ flat.  

\noindent
ii)\;
From \eqref{eq:lcunimod} and identities of classical vector algebra:
\begin{equation}
\norm{\nabla\s}^2
=\ts\sum_i\abs{\nab{\s_i}\s}^2
=\sum_i\mu_i^{\;2}\/\abs{\s_i\times\s}^2
=\sum_i\mu_i^{\;2}(\abs{\s}^2-a_i^{\;2}).
\label{eq:normsq}
\end{equation}
In the non-abelian flat case $\mu_1=\mu_2=0$ but $\mu_3\neq0$ (according to `Milnor's list'), so $\s$ is parallel if and only if $\s$ is a multiple of $\s_3$, whereas $[\fg,\fg]$ is generated by $\s_1$ and $\s_2$.  
\end{proof}

It follows from Corollary \ref{cor:hphasecsb} that Theorem \ref{thm:3dflat} identifies all the invariant harmonic vector fields on $G$; equivalently, the invariant zeroes of vertical energy.  We now extend this to higher-power vertical energy.  To clarify terminology:

\begin{defn}\label{defn:ricciflat}
An invariant vector field $\s$ is \textsl{Ricci-flat} if $\s\in\fn$ (rather than the weaker condition $\Ric(\s,\s)=0$).
\end{defn}

\begin{thm}\label{thm:3dzero}
With the same hypotheses as Theorem \ref{thm:3dflat}, $\s$ is a $r$-harmonic vector field in precisely the following cases:

\begin{itemize}[leftmargin=1.5em, itemsep=0.0ex]
\item[\rm a)]
$r=1\colon$ 
$(G,g)$ is flat and $\s$ is orthogonal to $[\fg,\fg]$.

\item[\rm b)]
$r=2\colon$ 
$\s$ is Ricci-flat.

\item[\rm c)]
$r=3\colon$ 
all $\s$.
\end{itemize}
\end{thm}

\begin{proof}
Case (a) follows from Theorem \ref{thm:3dflat}\,(ii), whereas (c) follows from Corollary \ref{cor:hphasecsb} since $TG$ has rank 3.  By Corollary \ref{cor:hphasecsb}, $\s$ is a $2$-harmonic vector field precisely when $\s$ is $2$-parallel, and using \eqref{eq:lcunimod} and \eqref{eq:prinric} along with various well-known identities of classical vector algebra we compute:
\begin{align}
\norm{\nabla\s\wedge\nabla\s}^2
&=\ts\sum_{i<j}
\bigl(\abs{\nab{\s_i}\s}^2 \abs{\nab{\s_j}\s}^2
-\<\nab{\s_i}\s,\nab{\s_j}\s\>^2\bigr)  
\notag \\[1ex]
&=\ts\sum_{i<j}\mu_i^{\;2}\mu_j^{\;2}
\bigl(\abs{\s_i\times\s}^2\,\abs{\s_j\times\s}^2
-\<\s_i\times\s,\s_j\times\s\>^2\bigr) 
\notag \\[1ex]
&=\ts\sum_{i<j}\mu_i^{\;2}\mu_j^{\;2}
\abs{(\s_i\times\s)\times(\s_j\times\s)}^2 
\notag \\[1ex]
&=\ts\sum_{i<j}\mu_i^{\;2}\mu_j^{\;2}
\<\s,\s_i\times\s_j\>^2\/\abs{\s}^2 
\notag \\[1ex]
&=\tfrac14(\rho_1^{\;2}a_1^{\;2}
+\rho_2^{\;2}a_2^{\;2}
+\rho_3^{\;2}a_3^{\;2})\abs{\s}^2 
=\tfrac14 \abs{\s}^2\,\abs{\Ric(\s)}^2.
\label{eq:2wedge}
\end{align}
Thus $\s$ is 2-parallel precisely when $\Ric(\s)=0$.
\end{proof}

\begin{rems}\label{rems:3dzero}
\item{}
\begin{itemize}[leftmargin=1.4em]
\item[1)]
Non-trivial $2$-harmonic invariant vector fields exist only if the Ricci curvature degenerates.  For example, it follows from `Milnor's list' that there are no such vector fields if $\fg=\nil$, or $\fg=\fe(2)$ with non-flat metric.

\item[2)]
If $\s$ is Ricci-flat then the image of $\nabla\s$ is orthogonal to $\fn$, which since $\dim\fn\neq1$ therefore lies in a fixed (ie.\ independent of $\s$) rank 1 subbundle.  This geometric `quirk' explains the linearity of the space of $2$-harmonic invariant vector fields, contrary to general expectation (Remark \ref{rem:rpowerspace}).
\end{itemize}
\end{rems}

We now consider the invariant $r$-harmonic unit vector fields.  When $r=1$ these were classified in \cite{GDV}, which we revisit in Theorem \ref{thm:1pharmsb} below. When $r=3$ the classification is tautologous: the $3$-rd vertical energy is zero since the fibres of $UG$ are $2$-dimensional. The outstanding case ($r=2$) requires the $1$-st vertical Newton tensor, and its divergence.

\begin{thm}\label{thm:newton}
Let $\newt=\nu_1^v(\s)$ for an invariant unit vector field $\s$. 
For all $\v\in\fg$:

\begin{itemize}[leftmargin=1.5em, itemsep=1ex]
\item[\rm i)]
$\newt(\v)=\norm{M}^2\/\v-\v^{(2)}+\<\v,\s^{(1)}\>\s^{(1)}-\abs{\s^{(1)}}^2\v$.

\item[\rm ii)]
$\diverge \newt=\s^{(2)}\times \s^{(1)}$.
\end{itemize}

\noindent
Then $\newt$ is solenoidal in precisely the following situations:

\begin{itemize}[leftmargin=1.5em, itemsep=0.5ex]
\item[\rm a)]
$\s$ is a principal structure direction;

\item[\rm b)]
$\s$ is flat;

\item[\rm c)]
$\s$ lies in a principal section orthogonal to the Ricci kernel $\fn$, when $\dim\fn=2$.
\end{itemize}
\end{thm}

\begin{note*}
Conditions (a), (b) and (c) are not mutually exclusive.
\end{note*}

\begin{proof}  
\item{}
i)\;
From \eqref{eq:newtrec}:
\begin{equation}
\newt
=\vrs_1^v(\s)-\a,
\label{eq:firstvnewt}
\end{equation}
where $\a$ is the vertical Cauchy-Green tensor.  By \eqref{eq:vcgvb} and \eqref{eq:lcunimod}, for all $\v,\psi\in\fg$:
\begin{align*}
\<\a(\v),\psi\>
=\<\nab\v\s,\nab\psi\s\>
&=\<\v^{(1)}\times\s,\psi^{(1)}\times\s\> \\
&=\<\v^{(1)},\psi^{(1)}\>-\<\v^{(1)},\s\>\<\psi^{(1)},\s\>,
\quad\text{since $\abs{\s}=1$} \\
&=\<\v^{(2)},\psi\>-\<\v,\s^{(1)}\>\<\s^{(1)},\psi\>.
\end{align*}
Hence:
\begin{equation}
\a(\v)=\v^{(2)}-\<\v,\s^{(1)}\>\s^{(1)}.
\label{eq:cgunimod}
\end{equation}
Furthermore, from \eqref{eq:srvbdensity} and \eqref{eq:normsq}:
\begin{equation}
\vrs_1^v(\s)
=\norm{\nabla\s}^2
=\norm{M}^2-\abs{\s^{(1)}}^2.
\label{eq:bend}
\end{equation}

\smallskip\noindent
ii)\;
Since the $\s_i$ are geodesic, by Definition \ref{defn:twistrace}:
\begin{align*}
\diverge \newt
=\ts\sum_i\nab{\s_i}(\newt \s_i)
&=\ts\sum_i \mu_i\s_i\times \newt \s_i,
\quad\text{by \eqref{eq:lcunimod}} 
\\[0.5ex]
&=\ts\sum_i \<\s_i,\s^{(2)}\>(\s_i\times\s^{(1)}),
\quad\text{by (i)} \\[0.5ex]
&=\s^{(2)}\times\s^{(1)}.
\end{align*}
It follows that $\diverge \newt=0$ if and only if $M\s$ is an eigenvector of $M$.  Since $M$ has the same eigenspaces as $L$, this is the case if $\s$ is a principal structure direction, and the converse holds if the Ricci curvature is non-degenerate (since $M$ is then invertible).  At the other extreme, if $(G,g)$ is flat then at least two Milnor numbers vanish, so $\diverge \newt=0$ for all $\s$; by Theorem \ref{thm:3dflat}\,(i) this is the case precisely when $\s$ is flat.  If $\dim\fn=2$ then exactly one Milnor number vanishes, say $\mu_k$, and $\diverge \newt=0$ precisely when $\s$ lies in a principal section containing $\s_k$. 
\end{proof}

The following expression for the second covariant derivative is a straightforward consequence of equations \eqref{eq:2cov} and \eqref{eq:lcunimod}. 

\begin{lem}\label{lem:2deriv}
For all $\v,\psi\in\fg$ we have:
$$
\nabsq{\v}{\psi}\s
=\<\v^{(1)},\s\>\psi^{(1)}-\<\v^{(1)},\psi^{(1)}\>\s
-(\v^{(1)}\times\psi)^{(1)}\times\s.
$$
\end{lem}

\begin{thm}\label{thm:2pharmsb}
Under the same hypotheses as Theorem \ref{thm:3dzero}, with $\s$ of unit length:
$$
-4\,\T_2(\s)=\abs{\Ric(\s)}^2\s+\Ric^2(\s),
$$
where $\Ric^2$ denotes the iterated Ricci endomorphism.  Then $\s$ is a $2$-harmonic unit vector field precisely when any of the following hold:
\begin{itemize}[leftmargin=1.5em, itemsep=0.5ex]
\item[\rm a)]
$\s$ is a principal structure direction;

\item[\rm b)]
$\s$ lies in a $2$-dimensional subalgebra;

\item[\rm c)]
$\s$ is Ricci-flat.
\end{itemize}
\end{thm}

\begin{proof}
We develop the expression for $\T_2(\s)$ from Theorem \ref{thm:hphasecvb}. By Theorem \ref{thm:newton}\,(i):
$$
\tr{\newt}\nabla^2\s
=\norm{M}^2 \trace\nabla^2\s
-\tr{M^2}\nabla^2\s
+\nabsq{M\s}{M\s}\s
-\abs{\s^{(1)}}^2\trace\nabla^2\s.
$$
Now by Lemma \ref{lem:2deriv}:
$$
\nabsq{\s_i}{\s_i}\s
=\mu_i^{\;2}a_i\/\s_i-\mu_i^{\;2}\s, 
$$
so:
\begin{align}
\trace\nabla^2\s
&=\ts\sum_i \nabsq{\s_i}{\s_i}\s
=\s^{(2)}-\norm{M}^2\/\s,
\label{eq:roughlap}
\intertext{and:}
\tr{M^2}\nabla^2\s
&=\ts\sum_i \mu_i^{\,2}\,\nabsq{\s_i}{\s_i}\s
=\s^{(4)}-\norm{M}^4\/\s+2\/\vrs_2(M^2)\/\s,
\notag
\end{align}
where $\vrs_2(M^2)$ is simply the second elementary symmetric polynomial in $\mu_1^{\;2},\mu_2^{\;2},\mu_3^{\;2}$.
By Lemma \ref{lem:2deriv} again:
$$
\nabsq{M\s}{M\s}\s
=\abs{\s^{(1)}}^2\,\s^{(2)}
-\abs{\s^{(2)}}^2\/\s
-(\s^{(2)}\times\s^{(1)})^{(1)}\times\s.
$$ 
Applying Theorem \ref{thm:newton}\,(ii) and equation \eqref{eq:lcunimod} yields:
$$
\T_2(\s)
=B(\s)\/\s+\norm{M}^2\/\s^{(2)}-\s^{(4)},
$$
where:
$$
B(\s)=\norm{M}^2\,\abs{\s^{(1)}}^2-\abs{\s^{(2)}}^2
-2\/\vrs_2(M^2),
$$
after (somewhat remarkably) three pairs of cancellations, including all terms involving $\norm{M}^4$ and $\diverge \newt$.  Now:
\begin{align*}
B(\s)
&=\ts\sum_{i,j} \mu_i^{\;2}\mu_j^{\;2} a_j^{\;2}
-\ts\sum_i \mu_i^{\;4} a_i^{\;2}
-\ts\sum_{i\neq j} \mu_i^{\;2}\mu_j^{\;2} \\[0.5ex]
&=\ts\sum_{i\neq j} \mu_i^{\;2}\mu_j^{\;2}(a_j^{\;2}-1) \\[0.5ex]
&=\tfrac14\ts\sum_{i\neq j} \rho_i^{\;2}(a_j^{\;2}-1),
\quad\text{by \eqref{eq:prinric}} \\[0.5ex]
&=-\tfrac14\ts\sum_i \rho_i^{\;2}(a_i^{\;2}+1),
\quad\text{since $a_1^{\;2}+a_2^{\;2}+a_3^{\;2}=1$} \\[0.5ex]
&=-\tfrac14\norm{\Ric}^2-\tfrac14\abs{\Ric(\s)}^2.
\end{align*}
Furthermore:
\begin{align*}
\norm{M}^2\/\s^{(2)}-\s^{(4)}
&=\ts\sum_{i\neq j} \mu_i^{\;2}\mu_j^{\;2}a_j\/\s_j 
=\tfrac14 \ts\sum_{i\neq j} \rho_i^{\;2}a_j\/\s_j \\[0.5ex]
&=\tfrac14(\ts\sum_i \rho_i^{\;2})\s
-\tfrac14\ts\sum_i\rho_i^{\;2}a_i\/\s_i \\[0.5ex]
&=\tfrac14\norm{\Ric}^2\s-\tfrac14\Ric^2(\s).
\end{align*}
Terms involving $\norm{\Ric}$ cancel, leaving the stated formula for $\T_2(\s)$.  

\par
It follows from Theorem \ref{thm:hphasecsb} that $\s$ is a 2-harmonic unit vector field precisely when $\s$ is an eigenvector of $\Ric^2$.  The characterisation of the eigenspaces of $\Ric^2$ follows from the identities:
\begin{equation}
\rho_i^{\;2}-\rho_j^{\;2}
=4(\mu_j^{\;2}-\mu_i^{\;2})\mu_k^{\;2}
\label{eq:modrho}
\end{equation}
and:
\begin{equation}
\mu_i^{\;2}-\mu_j^{\;2}
=(\lam_j-\lam_i)\lam_k,
\label{eq:modmu}
\end{equation}
for $\{i,j,k\}=\{1,2,3\}$, bearing in mind equations \eqref{eq:prinric} and \eqref{eq:commrel}.
\end{proof}

For comparison, we give a similar characterisation of the invariant harmonic unit vector fields on $G$.   This is simply a consequence of equation \eqref{eq:roughlap} in the proof of Theorem \ref{thm:2pharmsb}, and Theorem \ref{thm:hphasecsb} (see also \cite[Lemma 5.1]{GDV}), along with \eqref{eq:modmu}.

\begin{thm}\label{thm:1pharmsb}
Under the same hypotheses as Theorem \ref{thm:2pharmsb}:
$$
\T_1(\s)=M^2(\s)-\norm{M}^2\s,
$$
where $M^2$ is the iterated Milnor map.
Then $\s$ is a  harmonic unit vector field precisely when either of the following hold:
\begin{itemize}[leftmargin=1.5em, itemsep=0.5ex]
\item[\rm a)]
$\s$ is a principal structure direction.

\item[\rm b)]
$\s$ lies in a $2$-dimensional subalgebra.
\end{itemize}
\end{thm}

Let $\S\subset\fg$ be the unit sphere, and for $r=1,2,3$ denote by $\H_r\subseteq\S$ the invariant $r$-harmonic unit vector fields, and by $\Z_r\subseteq\H_r$ those of zero $r$-th vertical energy.  Then $\H_3=\Z_3$, and comparing Theorems \ref{thm:3dzero}, \ref{thm:2pharmsb} and \ref{thm:1pharmsb} for $r=2$ yields:

\begin{cor}\label{cor:harmz}
$\H_r=\H_{r-1}\cup\Z_r$ for $r=2,3$.
\end{cor}

It follows from Theorem \ref{thm:2pharmsb} (resp.\ Theorem \ref{thm:1pharmsb}) that the invariant $r$-harmonic unit vector fields on $(G,g)$ are determined by the absolute values of the principal Ricci curvatures (resp.\ Milnor numbers) when $r=2$ (resp.\ $r=1$).   From equations \eqref{eq:modrho} and
\eqref{eq:modmu}, if the $\abs{\rho_i}$ are distinct then so are the $\lam_i$, so the polar set:
$$
\P=\{\pm\s_1,\pm\s_2,\pm\s_3\}
$$  
is well-defined.  Furthermore if $\abs{\rho_k}$ is distinct from $\abs{\rho_i}$ and $\abs{\rho_j}$ for $\{i,j,k\}=\{1,2,3\}$ then $\lam_k$ is distinct from $\lam_i$ and $\lam_j$, so the polar pair $\P_k=\{\pm\s_k\}$ and corresponding equatorial circle:
$$
\C_{ij}=\{a_i\s_i+a_j\s_j:a_i^{\,2}+a_j^{\,2}=1\}
$$
are well-defined.  Theorems \ref{thm:3dzero}, \ref{thm:2pharmsb} and \ref{thm:1pharmsb} may now be summarised as follows:

\begin{cor}\label{cor:harm12}

\item{}
\begin{itemize}[leftmargin=1.6em, itemsep=1ex]
\item[\rm i)]
Suppose the $\abs{\rho_i}$ are distinct.  Then $\H_1=\H_2=\P$ and $\Z_1=\Z_2=\eset$.

\item[\rm ii)]
Suppose $\abs{\rho_i}=\abs{\rho_j}\neq\abs{\rho_k}$.  Then $\H_1=\H_2=\C_{ij}\cup\P_k$ and $\Z_1=\Z_2=\eset$, unless $\rho_i=\rho_j=0$ and $\abs{\mu_i}\neq\abs{\mu_j}$ in which case $\H_1=\P$ and $\Z_2=\C_{ij}$.

\item[\rm iii)]
Suppose $\abs{\rho_1}=\abs{\rho_2}=\abs{\rho_3}$.  Then $\H_1=\H_2=\S$ and $\Z_1=\Z_2=\eset$, unless $(G,g)$ is flat in which case $\H_1=\C_{ij}\cup\P_k$ with $\Z_1=\P_k$ and $\Z_2=\S$ if $G$ is non-abelian with $\mu_k\neq0$, or $\Z_1=\Z_2=\S$ if $G$ is abelian.
\end{itemize}
\end{cor}

By placing Corollary \ref{cor:harm12} alongside `Milnor's list', and noting from \eqref{eq:modmu} that $\abs{\mu_i}=\abs{\mu_j}$ if and only if $\lam_i=\lam_j$ or $\lam_k=0$, we obtain the following scheme, which should be compared to  that of \cite[Prop.\ 5.2]{GDV}.

\smallskip
\begin{itemize}[leftmargin=3.9em, itemsep=1ex]

\item[$\fa\colon$] 
$\Z_1=\Z_2=\H_1=\H_2=\S$. 

\item[$\nil\colon$]
$\H_1=\H_2=\S$ and $\Z_1=\Z_2=\eset$.

\item[$\fe(1,1)\colon$] 
$\H_1=\H_2=\C_{13}\cup\P_2$ and $\Z_1=\Z_2=\eset$, unless $\lam_1=-\lam_3$ in which case $\Z_2=\C_{13}$.

\item[$\fe(2)\colon$] 
$\H_1=\H_2=\C_{12}\cup\P_3$ and $\Z_1=\Z_2=\eset$, unless
$\lam_1=\lam_2$ in which case $\Z_2=\H_2=\S$ and $\Z_1=\P_3$.

\item[$\fs\fl(2)\colon$]
$\H_1=\H_2=\P$ and $\Z_1=\Z_2=\eset$, unless:

\par
$\lam_1=\lam_2$ in which case $\H_1=\H_2=\C_{12}\cup\P_3$;

\par
$\lam_1=\lam_2-\lam_3$ in which case $\H_2=\C_{13}\cup \P_2$ and $\Z_2=\C_{13}$.

\item[$\fs\fu(2)\colon$]
$\H_1=\H_2=\P$ and $\Z_1=\Z_2=\eset$, unless: 

\par
$\lam_1=\lam_2=\lam_3$ in which case $\H_1=\H_2=\S$;

\par
$\lam_1=\lam_2>\lam_3$ in which case $\H_1=\H_2=\C_{12}\cup\P_3$; 

\par
$\lam_1>\lam_2=\lam_3\neq \tfrac12\lam_1$ in which case $\H_1=\H_2=\C_{23}\cup\P_1$;

\par
$\lam_2=\lam_3=\tfrac12\lam_1$ in which case $\H_1=\H_2=\C_{23}\cup\P_1$ and $\Z_2=\C_{23}$;

\par
$\lam_1=\lam_2+\lam_3$ and $\lam_2\neq\lam_3$ in which case $\H_2=\C_{23}\cup\P_1$ and $\Z_2=\C_{23}$.  
\end{itemize}

\begin{rem}
Left-invariant metrics with $\H_1\subsetneq\H_2$ are supported (only) on $\fe(2)$, $\fs\fl(2)$ and $\fs\fu(2)$, and in each case $\H_2=\H_1\cup\Z_2$, as expected.
\end{rem}

Since $\H_1\subseteq\H_2$, all invariant harmonic unit vector fields are twisted $2$-skyrmions in  the bundle $UG\to G$ (Definition \ref{defn:twistskyr}), for all coupling constants.  In fact there are no others.

\begin{thm}
\label{thm:uniskyrme}
An invariant unit vector field $\s$ is a twisted $2$-skyrmion in the unit tangent bundle precisely when $\s\in\H_1$.
\end{thm}

\begin{proof}
By equation \eqref{eq:firstvarvb} for the first variation of higher-power vertical energy of sections of vector bundles, and the argument used to prove Theorem \ref{thm:hphasecsb}, $\s$ is a twisted $2$-skyrmion if and only if $\T_1(\s)+c\,\T_2(\s)$ is a pointwise multiple of $\s$ for some $c>0$, and by Theorems \ref{thm:2pharmsb} and \ref{thm:1pharmsb} this is the case precisely when $\s$ is an eigenvector of:
$$
M^2-\frac{c}{4}\Ric^2.
$$
Now:
$$
M^2(\s)-\frac{c}{4}\Ric^2(\s)
=\ts\sum_i \eta_i\/a_i\/\s_i,
$$ 
where by \eqref{eq:prinric}: 
$$
\eta_i=\mu_i^{\;2}-c\/\mu_j^{\;2}\mu_k^{\;2},
$$
for $\{i,j,k\}=\{1,2,3\}$.  Then:
$$
\eta_i-\eta_j=(\mu_i^{\;2}-\mu_j^{\;2})(1+c\/\mu_k^{\;2}),
$$
hence $\eta_i=\eta_j$ if and only if $\abs{\mu_i}=\abs{\mu_j}$, so the classification scheme is identical to that for $\H_1$.
\end{proof}

We conclude with a classification of all invariant higher-power harmonic maps $G\to UG$.  This will ultimately require the $2$-nd vertical Newton tensor and its divergence.

\begin{lem}\label{lem:vnewt2}
Suppose $\s\in\H_1$.  Then for all $\v\in\fg$:

\begin{itemize}[leftmargin=1.5em, itemsep=1ex]
\item[\rm i)]
$\nu_2^v(\s)\colon\v
\mapsto\v^{(4)}-\vrs_1^v(\s)\/\v^{(2)}
+\vrs_2^v(\s)\v
+(\vrs_1^v(\s)-\abs{\s^{(1)}}^2)\<\v,\s^{(1)}\>\s^{(1)}$.

\item[\rm ii)]
$\diverge \nu_2^v(\s)=(\vrs_1^v(\s)-\abs{\s^{(1)}}^2)
\diverge \nu_1^v(\s)$.
\end{itemize}
\end{lem}

\begin{proof}
Since $\s\in\H_1$ it follows from Theorem \ref{thm:1pharmsb} that $\s$ is an eigenvector of $M^2$:
\begin{equation}
\s^{(2)}=\abs{\s^{(1)}}^2\/\s.
\label{eq:vharm1}
\end{equation}
From \eqref{eq:newtrec}:
$$
\nu_2^v(\s)=\vrs_2^v(\s)-\vrs_1^v(\s)\a+\a^2,
$$
where $\a$ is the vertical Cauchy-Green tensor.  From \eqref{eq:cgunimod}:
\begin{gather*}
\a(\v)
=\v^{(2)}-\<\v,\s^{(1)}\>\s^{(1)},
\intertext{hence after simplification using \eqref{eq:vharm1}:}
\begin{aligned}
\a^2(\v)
&=\v^{(4)}-\abs{\s^{(1)}}^2\<\v,\s^{(1)}\>\s^{(1)},
\end{aligned}
\end{gather*}
which yields the formula for $\nu_2^v(\s)$.  Now, since the $\s_i$ are geodesic:
\begin{align*}
\diverge \nu_2^v(\s)
=\ts\sum_i \nab{\s_i}(\nu_2^v(\s)\s_i)
&=(\vrs_1^v(\s)-\abs{\s^{(1)}}^2)
\ts\sum_i \<\s_i,\s^{(1)}\>\mu_i\s_i\times\s^{(1)} \\[0.5ex]
&=(\vrs_1^v(\s)-\abs{\s^{(1)}}^2)\/\s^{(2)}\times\s^{(1)},
\end{align*}  
and the result follows from Theorem \ref{thm:newton}\,(ii).
\end{proof}

\begin{thm}\label{thm:hpharmaps}
An invariant unit vector field $\s$ is a $r$-harmonic map $G\to UG$ if and only if $\s$ is a principal structure direction {\rm(}$r=1,2${\rm)}, or $\s\in\H_1$ {\rm(}$r=3${\rm )}.
\end{thm}

\begin{proof}
We work sequentially through the `powers', leveraging our results as we go.

\medskip\noindent
a) $r=1$.
It follows from Theorem \ref{thm:hasec} that $\s$ is a harmonic map precisely when $\s\in\H_1$ and  $d\pi\circ\tau(\s)=0$.  By Theorem \ref{thm:taurhoriz} and equation \eqref{eq:lcunimod}:
\begin{equation}
d\pi\circ\tau(\s)
=\ts\sum_i R(\s,\nab{\s_i}\s)\s_i
=\ts\sum_i \mu_i\/R(\s,\s_i\times\s)\s_i, 
\label{eq:dpitau}
\end{equation}
and by \eqref{eq:3driem}:
\begin{align}
R(\s,\s_i\times\s)\s_i
&=a_i\/R(\s_i,\s_i\times\s)\s_i
=\ts\sum_{j,k}\e_{ijk}\,a_i\/a_j\/R(\s_i,\s_k)\s_i 
\notag \\
&=-\ts\sum_{j,k}\e_{ijk}\,a_i\/a_j\/K_{ik}\,\s_k.
\label{eq:unimodtauhoriz1}
\end{align}
Now  $\s\in\H_1$ if and only if $\s$ is a principal structure direction or $\s$ lies in a $2$-dimensional subalgebra (Theorem \ref{thm:1pharmsb}).  If the former then $\s_i$ may be chosen so that two of the $a_i$ vanish, hence $d\pi\circ\tau(\s)=0$.  If the latter then by \eqref{eq:commrel} $\lam_k=0=a_k$ for some $k$, say for argument $k=1$, so by \eqref{eq:mui}:
\begin{equation}
\mu_2=\tfrac12(\lam_3-\lam_2)=-\mu_3.
\label{eq:mulam}
\end{equation}
Then by \eqref{eq:prinsec}, \eqref{eq:dpitau} and \eqref{eq:unimodtauhoriz1}:
\begin{equation}
d\pi\circ\tau(\s)
=-\mu_2\/a_2\/a_3\/(K_{12}+K_{13})\,\s_1
=2\/\mu_2^{\;3} a_2\/a_3\,\s_1,
\label{eq:dpitausalg}
\end{equation}
which by \eqref{eq:mulam} vanishes precisely when $\s$ is a principal structure direction.   

\medskip\noindent
b) $r=2$.
It follows from Theorems \ref{thm:hasec} and \ref{thm:uniskyrme} that  $\s$ is a $2$-harmonic map precisely when $\s\in\H_1$ and $d\pi\circ \tau_2(\s)=0$.  By Theorem \ref{thm:hasec}:
\begin{equation}
\nu_1(\s)=\nu_1^v(\s)+2,
\label{eq:newt1}
\end{equation}
hence by Theorem \ref{thm:taurhoriz}:
$$
d\pi\circ\tau_2(\s)
=\diverge \nu_1^v(\s)
+2\/d\pi\circ\tau(\s)
+\ts\sum_i R(\s,(\nabla\s)\circ\nu_1^v(\s)\s_i)\s_i.
$$
When $\s\in\H_1$ it follows from \eqref{eq:firstvnewt} and \eqref{eq:cgunimod} that:
$$
(\nabla\s)\circ\nu_1^v(\s)\s_i
=(\vrs_1^v(\s)-\mu_i^{\;2})\/\nab{\s_i}\s,
$$
since by \eqref{eq:lcunimod} and \eqref{eq:vharm1} the covariant derivative of $\s$ along $\s^{(1)}$ vanishes.  Therefore by Theorem \ref{thm:taurhoriz} and equation \eqref{eq:lcunimod}:
$$
d\pi\circ\tau_2(\s)
=(\vrs_1^v(\s)+2)\/d\pi\circ\tau(\s)
+\diverge \nu_1^v(\s)
-\ts\sum_i \mu_i^3\/R(\s,\s_i\times\s)\s_i.
$$
If $\s$ is a principal structure direction then $\nu_1^v(\s)$ is solenoidal by Theorem \ref{thm:newton}, $d\pi\circ \tau(\s)=0$ by (a), and $R(\s,\s_i\times\s)\s_i=0$ by \eqref{eq:unimodtauhoriz1}; hence $d\pi\circ\tau_2(\s)=0$.  If $\s$ lies in a $2$-dimensional subalgebra, say $\lam_1=0=a_1$ and consequently $\mu_2=-\mu_3$, then by Theorem \ref{thm:newton} and \eqref{eq:dpitausalg}:
\begin{equation}
\diverge \nu_1^v(\s)
=\s^{(2)}\times\s^{(1)}
=-2\/\mu_2^{\;3}a_2\/a_3\,\s_1
=-d\pi\circ\tau(\s).
\label{eq:divT}
\end{equation}
Furthermore by \eqref{eq:unimodtauhoriz1} and \eqref{eq:dpitausalg}:
\begin{equation}
\ts\sum_i \mu_i^3\/R(\s,\s_i\times\s)\s_i
=-\mu_2^{\;3}a_2\/a_3(K_{12}+K_{13})\,\s_1 
=\mu_2^{\;2}\/d\pi\circ\tau(\s).
\label{eq:mucurv}
\end{equation}
Therefore:
\begin{equation}
d\pi\circ\tau_2(\s)
=(\vrs_1^v(\s)-\mu_2^{\;2}+1)\/d\pi\circ\tau(\s).
\label{eq:dpitau2}
\end{equation}
By \eqref{eq:normsq}:
\begin{equation}
\vrs_1^v(\s)
=\mu_1^{\;2}+\mu_2^{\;2},
\label{eq:eps1v}
\end{equation}
hence:
$$
d\pi\circ\tau_2(\s)
=(1+\mu_1^{\;2})\/d\pi\circ\tau(\s),
$$
which by (a) vanishes precisely when $\s$ is a principal structure direction. 

\medskip\noindent
c) $r=3$.  
Since $\tau_3^v(\s)=0$ it follows from Theorems \ref{thm:hasec} and \ref{thm:uniskyrme} that $\s$ is a $3$-harmonic map precisely when $\s\in\H_1$ and $d\pi\circ\tau_3(\s)=0$. By Theorem \ref{thm:hasec}:
$$
\nu_2(\s)=\nu_2^v(\s)+\nu_1^v(\s)+1
=\nu_1(\s)-1+\nu_2^v(\s),
$$ 
hence by Theorem \ref{thm:taurhoriz}:
$$
d\pi\circ\tau_3(\s)
=d\pi\circ\tau_2(\s)-d\pi\circ\tau(\s)+\diverge \nu_2^v(\s)
+\ts\sum_i R(\s,(\nabla\s)\circ\nu_2^v(\s)\s_i)\s_i.
$$
When $\s\in\H_1$ it follows from Lemma \ref{lem:vnewt2} that:
$$
(\nabla\s)\circ\nu_2^v(\s)\s_i
=\bigl(\vrs_2^v(\s)-\vrs_1^v(\s)\mu_i^{\;2}+\mu_i^{\;4}\bigr)\nab{\s_i}\s.
$$
Therefore by Lemma \ref{lem:vnewt2} again, Theorem \ref{thm:taurhoriz} and equation \eqref{eq:lcunimod}:
\begin{align*}
d\pi\circ\tau_3(\s)
&=d\pi\circ\tau_2(\s)+(\vrs_2^v(\s)-1)\/d\pi\circ\tau(\s)
+(\vrs_1^v(\s)-\abs{\s^{(1)}}^2)\diverge \nu_1^v(\s) \\
&\qquad
-\ts\sum_i (\vrs_1^v(\s)-\mu_i^{\;2})\mu_i^{\;3}R(\s,\s_i\times\s)\s_i.
\end{align*}
If $\s$ is a principal structure direction then each summand vanishes by (a) and (b).
If $\s$ lies in a $2$-dimensional subalgebra, with $\lam_1=0=a_1$ and $\mu_2=-\mu_3$, then:
$$
\abs{\s^{(1)}}^2
=\ts\sum_i \mu_i^{\;2}\/a_i^{\;2}
=\mu_2^{\;2},
$$ 
hence by \eqref{eq:dpitau2}, \eqref{eq:divT}, and \eqref{eq:mucurv}:
$$
d\pi\circ\tau_3(\s)
=C(\s)\,d\pi\circ\tau(\s),
$$
where:
$$
C(\s)=\vrs_2^v(\s)-\vrs_1^v(\s)\/\mu_2^{\;2}
+\mu_2^{\,4}.
$$
By \eqref{eq:2wedge} and \eqref{eq:prinric}:
$$
\vrs_2^v(\s)
=\tfrac14\abs{\Ric(\s)}^2
=\mu_1^{\;2}\mu_2^{\;2},
$$
hence by \eqref{eq:eps1v}:
\begin{equation*}
C(\s)
=\mu_1^{\;2}\mu_2^{\;2}-(\mu_1^{\;2}+\mu_2^{\;2})\mu_2^{\;2}
+\mu_2^{\;4}=0.
\qedhere
\end{equation*}
\end{proof}

\begin{rem}
The characterisation of invariant harmonic maps $\s\colon G\to UG$ agrees with that of \cite[Thm.\ 5.2]{GDV}.  Furthermore it was shown in \cite[Prop.\ 3.1]{TV} and \cite[Cor.\ 5.3]{GDV} that $\s$ is a minimal immersion precisely when $\s$ is a harmonic unit field; ie. $\s\in\H_1$.
\end{rem}

\end{document}